\documentclass[a4paper,reqno]{amsart}

\usepackage[a4paper,width={16cm},left=2.5cm,bottom=3cm, top=3cm]{geometry}
\usepackage{enumerate}
\usepackage{yfonts}

\usepackage[english]{babel}

\usepackage{a4wide}
\usepackage[dvipsnames]{xcolor}
\usepackage{amsmath}
\usepackage{amssymb}
\usepackage{amsfonts}
\usepackage{amsthm,graphicx}
\usepackage{comment}
\usepackage{csquotes}

\usepackage{hyperref}
\hypersetup{
	linktocpage=true,
	colorlinks=true,
	linkcolor=blue!70!black,
	citecolor=orange!40!red,
	urlcolor=magenta!80!black,
}

\usepackage{mathtools}

\numberwithin{equation}{section}
\newtheorem{theorem}{Theorem}[section]
\newtheorem{lemma}[theorem]{Lemma}

\newtheorem{proposition}[theorem]{Proposition}
\theoremstyle{definition}

\newtheorem{remark}[theorem]{Remark}

\newcommand{\mc}{\mathcal}
\newcommand{\mb}{\mathbb}

\newcommand{\la}{\lambda}
\newcommand{\norm}[1]{\left\lVert#1\right\rVert}
\newcommand{\pd}[2]{\frac{\partial#1}{\partial#2}}

\newcommand{\R}{\mb{R}}

\newcommand{\N}{\mb{N}}

\newcommand{\e}{\varepsilon}

\newcommand{\dive}{\mathop{\rm{div}}}

\newcommand{\Si}{\Sigma}
\newcommand{\Ee}{\mathsf{E}_\e}

\newcommand{\RN}{\R^N}
\newcommand{\dx}{\,{\rm d}x}
\newcommand{\ds}{\,{\rm d}\sigma}
\newcommand{\te}{\mc{T}_\e}
\newcommand{\tr}{\widetilde{\mathcal{T}}}

\usepackage{subfig}
\usepackage{tikz}
\usepackage{xcolor}
\usepackage{mathrsfs}
\usetikzlibrary{arrows}
\usetikzlibrary{shadings}
\makeatletter

\begin{document}
\title[Quantitative Spectral Stability for the Robin Laplacian]
{Quantitative Spectral Stability for the Robin Laplacian}

\author[V. Felli, P. Roychowdhury,  and G. Siclari]{Veronica Felli, Prasun Roychowdhury, and Giovanni Siclari}

\address{Veronica Felli and Prasun Roychowdhury
  \newline \indent Dipartimento di Matematica e Applicazioni
  \newline \indent
Universit\`a degli Studi di Milano–Bicocca
\newline\indent Via Cozzi 55, 20125 Milano, Italy}
\email{veronica.felli@unimib.it, prasun.roychowdhury@unimib.it}

\address{Giovanni Siclari 
  \newline \indent Dipartimento di Matematica
  \newline \indent Politecnico di Milano
  \newline\indent Piazza Leonardo da Vinci 32, 20133 Milano, Italy}
\email{ giovanni.siclari@polimi.it}

\begin{abstract}
  This paper deals with eigenelements of the Laplacian in bounded
  domains, under Robin boundary conditions, without any assumption on
  the sign of the Robin parameter. We quantify the asymptotics of the
  variation of simple eigenvalues under the singular perturbation
  produced by removing a shrinking set and imposing the same Robin
  condition on its boundary. We also study the convergence rate of the
  corresponding eigenfunctions.
\end{abstract}

\maketitle

\hskip10pt{\footnotesize {\bf Keywords.} Spectral stability, asymptotics
  of eigenvalues, blow-up analysis, Robin boundary conditions.

\medskip

}

\section{Introduction}\label{sec_intro}
Let $N\ge 2$ and $\Omega \subset \R^N$ be a bounded,
connected, and  Lipschitz open set.
We deal with Robin eigenvalue problems of the type
\begin{equation}\label{prob_robin}
\begin{cases}
-\Delta u= \la u, & \text{ in } \Omega,\\
\pd{u}{\nu} + \alpha u=0, & \text{ on } \partial \Omega,
\end{cases}
\end{equation}
where $\alpha \in \R$ and $\nu$ denotes the outer unit normal vector
to $\partial\Omega$.  Let $x_0\in \Omega$, $r_0>0$, and let
$\Sigma\subset\R^N$ be an open, bounded, and Lipschitz set such that
\begin{equation}\label{def_Sigma}
  \overline{B_{r_0}(x_0)}\subset\Omega,\quad 
  \overline{\Si} \subset B_{r_0},\quad \text{and}\quad
  B_{r_0}\setminus\overline{\Sigma}\text{ is connected},
\end{equation}
 where $B_{r_0}=\{y\in \R^N:|y|<r_0\}$ and $B_{r_0}(x_0)=x_0+B_{r_0}$.
For every $\e \in (0,1]$ we consider the scaled set 
\begin{equation}\label{def_Sigma_e}
  \Sigma_\e:=x_0+\e \Sigma.
\end{equation}
The aim of the present paper is to investigate the stability of the
spectrum of problem \eqref{prob_robin} under the removal of the set
$\overline{\Sigma_\e}$ from $\Omega$, as the parameter $\e \in (0,1]$
goes to $0$ and the hole $\overline{\Sigma_\e}$ consequently
disappears.  More precisely, we study the asymptotic behavior, as
$\e \to 0^+$, of the eigenvalues of the problem
\begin{equation}\label{prob_robin_perturbed}
\begin{cases}
-\Delta u= \la u, & \text{ in } \Omega\setminus\overline{\Sigma_\e},\\
\pd{u}{\nu} + \alpha u=0, & \text{ on } \partial \Omega,\\
\pd{u}{\nu} + \alpha u=0, & \text{ on } \partial \Sigma_\e,
\end{cases}
\end{equation}
proving, as a first step, that they converge to the eigenvalues of
problem \eqref{prob_robin}. While some results in this direction are
available in the literature, in this paper we are interested in
quantitative estimates, with the aim of computing the vanishing order
of the eigenvalue variation. We will focus on the case
\begin{equation}\label{eq:alpha-non-zero}
\alpha \neq 0
\end{equation}
since, while the choice $\alpha = 0$ could fall within our analysis as
a special case, it has already been analyzed in detail in
\cite{FLO_Neumann}.

In heat conduction theory, Robin boundary conditions serve to
interpolate between a perfectly insulating boundary, described by
Neumann boundary conditions $\alpha = 0$, and a temperature-fixing
boundary, represented by Dirichlet boundary conditions corresponding
to $\alpha= +\infty$. Notably, the Robin boundary condition is the
most commonly used in optical imaging, see
e.g. \cite{Aronson:95,Arridge20151033}.  Previous research has largely
focused on the initial few Robin eigenvalues, with particular emphasis
on applications in shape optimization, related isoperimetric
inequalities, and the asymptotic behavior of the lowest eigenvalues,
see e.g. \cite{BFK_survey}. The aim of the present paper is instead to
explore the stability of the spectrum of the Robin Laplacian amidst
singular perturbations, specifically the excision of small holes from
a bounded domain.  Given the significance of eigenvalues and
eigenvectors of differential operators in the theory of partial
differential equations, it is of interest to understand how
eigenelements are affected by small perturbations, such as changes in
the domain: this has relevance in various fields, including quantum
mechanics, materials science, heat conduction, climate modeling, and
acoustics.

Since the pioneering paper \cite{rauch_taylor:75}, a vast literature
about quantitative spectral stability for elliptic operators has
flourished. Geometric perturbations through domain perforation for the
Dirichlet Laplacian have been studied in
\cite{AFHC_spec_AB,ALM_cap_exp,
  FNO_disa_Dirichlet_region,O_Dirichlet_holes_euclidean} in Euclidean
spaces and in \cite{C_mainfols_holes} in the general setting of
Riemannian manifolds. The same kind of perturbations have been
investigated in \cite{AFN_fractional} for the fractional Laplacian
under Dirichlet boundary conditions and in \cite{FR_polyharmonic} for
polyharmonic operators, while the case of a generic Dirichlet form has
been addressed in \cite{BS_quantitative,McG_capacity}.  As discussed
in \cite{AFHC_spec_AB}, in dimension $N=2$ a relevant application of
the study of the Dirichlet Laplacian under removal of small slits can
be found, through some isospectrality property, in Aharonov-Bohm
operators with moving poles, see e.g.  \cite{FNOS_AB_12,FNS_AB} for
some recent quantitative spectral stability results.  Regarding
spectral stability under Neumann conditions on the boundary of small
holes, we refer to
\cite{FLO_Neumann,jimbo,LdC2012,nazarov2011,ozawa1983,ozawa1985}, in
addition to \cite{rauch_taylor:75}.

Eigenvalue problems in perforated domains have significant analogies
with problems with mixed varying boundary conditions, as highlighted
in \cite{FNO_disa_Dirichlet_region}, where a homogeneous Neumann
eigenvalue problem for the Laplacian is perturbed by prescribing zero
Dirichlet conditions on a small subset of the boundary.  We refer to
\cite{AO2023,FNO_disa_Neuman_region} for the complementary problem,
i.e.  the perturbation of the Dirichlet Laplacian by imposing a
homogeneous Neumann condition on a vanishing portion of the boundary.
Furthermore, such problems exhibit the same singular nature and
partially share the difficulties and methodological approaches even of
Dirichlet eigenvalue problems in domains with a thin tube attached to
the boundary, see \cite{AFT2014,AO2023,FO2020}.

  More closely related to the Robin Laplacian is the study of the
  dependence of the spectrum on $\alpha$ and the asymptotic behavior of
  eigenvalues as either $\alpha \to +\infty$ or $\alpha \to
  -\infty$. There are many results in this direction, see for example
  \cite{CCH_alpha,F_alpha,LP_alpha,O_Robin} and
  \cite{BFK_survey} for a comprehensive overview.  Regarding small
    perturbations produced by perforation, qualitative spectral
  stability in dimension $2$, under Dirichlet or Neumann conditions on
  $\partial \Omega$ and Robin conditions on the hole's boundary, has been
  established in \cite{BS2023}, in the case $\alpha>0$ (possibly
  depending on $\e$).  Under Dirichlet conditions on the external
  boundary of a planar domain and Robin conditions on the boundary of
  a removed shrinking disk, asymptotic expansions have been obtained in
  \cite{ozawa1,ozawa2}, letting the parameter $\alpha$ varying as
  $\kappa\e^\sigma$, with $\kappa>0$ and $\sigma\in\R$. The case of a large
  positive parameter $\alpha=\kappa\e^{-1}$, under Robin boundary
  conditions on both external and internal boundaries, has been
  treated in \cite{ward-henshaw-keller,ward-keller}, for a removed
  small subdomain of radius $O(\e)$.  Spectral stability for elliptic
  operators in perturbed domains subject to Robin boundary conditions
  has been investigated in \cite{DD1997}, for a wide class of domain
  perturbations, including cutting of small holes.  Finally, we
  mention \cite{BGT_domain_pert}  for spectral stability for the
  Laplacian under very general Robin conditions involving
  measures on the boundary, under some weak notion of convergence of sets.

  In the present paper, we first observe that the spectrum of the
  Robin problem is stable under the removal of small holes, i.e. the
  eigenvalues of the perturbed problem \eqref{prob_robin_perturbed}
  converge to those of the unperturbed one \eqref{prob_robin} as
  $\e\to0^+$, see Theorem \ref{theorem_spectral_stability}.  Although
  this result is well-known for $\alpha>0$, see e.g. \cite{DD1997} and
  \cite[Corollary 3.5]{BGT_domain_pert}, to the best of our knowledge
  it is new for $\alpha<0$.  Once it is understood that the eigenvalue
  variation due to the domain perturbation goes to zero as the hole
  disappears, we address the problem of quantifying its vanishing
  order for simple eigenvalues.  In Theorem
  \ref{theor_asympotic_eigen_precise_u_not_0}, we establish that such
  an order is $N-1$ if the point $x_0$, where the holes are
  concentrating, does not lie in the nodal set of the limit
  eigenfunction $u_n^0$.  On the other hand, if $u_n^0$ vanishes at
  $x_0$, the rate of convergence of the eigenvalues depends on the
  vanishing order of $u_n^0$ at $x_0$. This is proved in Theorem
  \ref{theor_asympotic_eigen_precise_u_0_dim_big} for $N\geq3$,
  exploiting the blow-up result provided by Theorem
  \ref{theorem_blow_up}, which allows us to obtain a sharp asymptotic
  expansion for the eigenvalue variation. In dimension $N=2$, the
  blow-up analysis fails to be effective; consequently, in Theorem
  \ref{theor_asympotic_eigen_precise_u_0_dim2} we are able to provide
  only an estimate for the vanishing rate of the eigenvalue variation,
  depending on the vanishing order of the limit eigenfunction.
  Nevertheless, in the case of a spherical hole we can derive a sharp
  expansion in any dimension $N\geq2$, providing a more explicit
  expression for its coefficients, see Theorems
  \ref{theor_asympotic_eigen_precise_round_u_0_dim_big} and
  \ref{theor_asympotic_eigen_precise_round_u_0_dim_2}.  Some results
  about the convergence of eigenfunctions are also obtained in
  Theorems \ref{theor_asympotic_eigenfunction_precise_u_0_dim_big} and
  \ref{theor_asympotic_eigenfunction_precise_u_0_dim_2}.

  The paper is organized as follows. In Section \ref{sec_main_results}
  we state and describe our main results. In Section \ref{sec_prelimi}
  we introduce the functional setting and discuss some preliminary
  trace and extension results; furthermore, we prove Theorem
  \ref{theorem_spectral_stability} establishing qualitative spectral
  stability. In Section \ref{sec_eign_var} an asymptotic expansion for
  the eigenvalue variation is obtained by applying the abstract result
  stated in Appendix \ref{sec_appendix}.  This analysis is then
  refined in Section \ref{sec_blow_up} by means of a blow up
  procedure. Furthermore, Section \ref{sec_round_hole} is focused on
  the case of round holes. Finally, in Appendix
  \ref{sec_appendix-extension} we give the details of the extension
  results presented in Section \ref{sec_prelimi} pointing out the
  relevance of the connectedness assumptions.

\section{Main results}\label{sec_main_results}
By classical spectral theory, see also Section \ref{sec_prelimi},
problem \eqref{prob_robin} admits a non-decreasing, diverging sequence
of eigenvalues $\{\la_j^0\}_{j \in \mathbb{N}\setminus\{0\}}$,
repeated according to their multiplicity.  Similarly, for every
$\e \in (0,1]$, \eqref{prob_robin_perturbed} admits a non-decreasing,
diverging sequence of eigenvalues
$\{\la_j^\e\}_{j \in \mathbb{N}\setminus\{0\}}$, counted with their
multiplicity.

The following spectral stability result serves as the starting point
of our analysis.
\begin{theorem}\label{theorem_spectral_stability}
For every $j \in \mb{N}\setminus\{0\}$
\begin{equation}\label{eq_limit_eigen}
  \lim_{\e \to 0^+}\la_j^\e= \la_j^0. 
\end{equation}
\end{theorem}
Once the above spectral stability is achieved, our main goal is to
quantify, for simple eigenvalues, the convergence rate of $\la_j^\e$
to $\la_j^0$. Let us fix $n\in\N\setminus\{0\}$ such that
\begin{equation}\label{eq:simple}
  \la_n^0\text{ is simple}.
\end{equation}
Let $u_n^0\in H^1(\Omega)$ be an associated eigenfunction
with
\begin{equation}\label{hp_u_n0}
\int_\Omega |u_n^0(x)|^2 \dx =1.
\end{equation}
By classical regularity theory, $u_n^0$ is smooth in
$\Omega$. For every $\e\in(0,1]$ we choose an
eigenfunction $u_n^\e$ associated
with $\la_n^\e$ such that
\begin{equation}\label{hp_u_ne}
  \int_{\Omega\setminus\overline{\Sigma_\e}} |u_n^\e(x)|^2 \dx =1
  \quad \text{ and } \quad \int_{\Omega\setminus\overline{\Sigma_\e}}
  u_n^0(x) u_n^\e(x) \dx \geq0.
\end{equation}
Different behaviors are observed in the two cases $u_n^0(x_0)\neq 0$
and $u_n^0(x_0)= 0$.
\begin{theorem}\label{theor_asympotic_eigen_precise_u_not_0}
Let $N\geq 2$ and $u_n^0(x_0)\neq 0$.  Then 
\begin{equation*}
  \la_n^\e -\la_n^0 = \alpha\, \mc{H}^{N-1}(\partial \Sigma)\, |u_n^0(x_0)|^2
  \e^{N-1}+ o(\e^{N-1})
  \quad \text{ as } \e \to 0^+,
\end{equation*}
where $\mc{H}^{N-1}$ is the $(N-1)$-dimensional
  Hausdorff measure.
\end{theorem}
An immediate consequence of the above theorem is that, if
$u_n^0(x_0)\neq 0$, for small $\e$ the variation $\la_n^\e -\la_n^0$
has the same sign as the parameter $\alpha$. In particular, if
$\alpha>0$, we have $\la_n^\e>\la_n^0$ provided $\e$ is sufficiently
small. This is consistent with the fact that, as $\alpha \to +\infty$,
the Robin problem tends to be a Dirichlet problem, whose eigenvalues
are monotone with respect to domain inclusion, in the sense that each
Dirichlet eigenvalue increases as the domain gets smaller.  It is
worth noting that, in the presence of small holes, the Robin problem
exhibits an order of convergence $N-1$ for the perturbed eigenvalues
which is intermediate between $N-2$, which is the order obtained in
\cite{AFHC_spec_AB} (see also \cite[Theorem
2.14]{FNO_disa_Dirichlet_region} and \cite [Appendix A]{ALM24}) for
the Dirichlet problem, and $N$, which is the one observed for the
Neumann problem in \cite{FLO_Neumann}. This behavior is naturally
expected, given the nature of the Robin operator as an intermediary
between the Dirichlet and the Neumann Laplacians.

In the case $u_n^0(x_0)= 0$, the situation is more intricate, since
other quantities, that vanish faster than $ \e^{N-1}$, come into play.
The vanishing order of the eigenvalue variation $\la_n^{\e}-\la_n^0$
can be quantified in terms of the asymptotic behavior of the 
eigenfunctions at $x_0$. For every 
$u\in C^\infty(\Omega)$ and $i\in \N$, we consider the
polynomial
\begin{align}\label{def_P_ui}
P^{u,x_0}_i(x):=\sum_{\beta\in \N^N, \, |\beta|=i}\frac{1}{\beta!}D^\beta
  u(x_0)x^\beta,\quad
  x\in\R^N,
\end{align}
where $|\beta|=\sum_{j=1}^N\beta_j$ denotes the length of the
multi-index $\beta=(\beta_1, \dots, \beta_N) \in \mb{N}^N$. In
particular, for $i=0$ we have $P^{u,x_0}_0(x):=u(x_0)$. We say that $u$ vanishes
of order $j\in \N$ at $x_0$ if
\begin{align*}
  P^{u,x_0}_j\not\equiv 0 \quad \text{and}\quad
  P^{u,x_0}_i\equiv 0 \quad  \text{for all } 0\leq i<j.
\end{align*}
In particular, if $u(x_0)\neq0$, the vanishing order of $u$ at $x_0$
is $0$.

From now on, we will denote by $k\geq 0$ the vanishing order of
$u_n^0$ at $x_0$.  We define
\begin{equation}\label{def_h}
 h:= \max\{k,1\}.
\end{equation}  
In dimension $N \ge 3$, we consider 
the unique finite energy weak  solution $\widetilde{V}$ to the problem  
\begin{equation}\label{srtel-rp}
\begin{dcases}
-\Delta \widetilde{V} = 0, & \text{in } \R^N\setminus \overline \Sigma, \\
\partial_\nu \widetilde{V} =  \partial_\nu P_h^{u_n^0,x_0}+\alpha\,
u_n^0(x_0), &
\text{on } \partial\Sigma,\\
\end{dcases}
\end{equation}
see Proposition \ref{prop_min_tilde_J}. The related quantity
\begin{equation}\label{def_torsion_global}
  \tr:= \int_{\R^N\setminus \Sigma}|\nabla \widetilde
  V|^2\dx=\int_{\partial \Sigma}
  \left(\partial_\nu P_h^{u_n^0,x_0} +\alpha\,
  u_n^0(x_0)\right) \widetilde{V}\ds
\end{equation}
plays a crucial role in the asymptotic expansion of perturbed
eigenelements.

The following theorem shows that, if $u_n^0(x_0)=0$ and $N\geq3$, the
quantity $\tr$ appears in the first significant term of the asymptotic
expansion of the eigenvalue variation.

\begin{theorem}\label{theor_asympotic_eigen_precise_u_0_dim_big}
  Let $N\geq 3$. If $u_n^0(x_0)=0$ and $k\geq1$ is the vanishing order
  of $u_n^0$ at $x_0$, then
\begin{equation*}
\la_n^\e -\la_n^0 = -\e^{N+2k-2}\left[\tr +
\int_{\Sigma}|\nabla P_k^{u_n^0,x_0}|^2\dx\right]+o(\e^{N+2k-2}) \quad \text{ as } \e \to 0^+.
\end{equation*}
\end{theorem}
Under the hypotheses of Theorem
\ref{theor_asympotic_eigen_precise_u_0_dim_big}, in particular under
the condition $u_n^0(x_0)=0$, we have $h=k$ and formula
\eqref{def_torsion_global} defining $\tr$ simplifies to
  \begin{equation*}
  \tr= \int_{\partial \Sigma} \big(\partial_\nu P_k^{u_n^0,x_0} \big) \widetilde{V}\ds.
\end{equation*}
In particular, in this case, the dominant term of
  the asymptotic expansion is the same as obtained in \cite{FLO_Neumann} for the
  Neumann problem.
Anyway, it is convenient to consider the more general definition of
$\tr$ given in \eqref{def_torsion_global}, which also contemplates the
case $u_n^0(x_0)\neq0$, since this will come into play in the
description of the variation of eigenfunctions even if
$u_n^0(x_0)\neq 0$, see Theorem
\ref{theor_asympotic_eigenfunction_precise_u_0_dim_big} below.

Theorem \ref{theor_asympotic_eigen_precise_u_0_dim_big} implies that,
if $N\geq3$, $u_n^0(x_0)=0$, and $\e$ is sufficiently small, the
eigenvalue variation $\la_n^\e -\la_n^0$ has a negative sign,
regardless of the sign of the parameter $\alpha$ and in contrast to
what happens in the Dirichlet case.

The proof of Theorem \ref{theor_asympotic_eigen_precise_u_0_dim_big}
is based on a blow-up analysis which does not seem to have a direct
generalization to the $2$-dimensional case; indeed, a Hardy type
inequality is used to identify the concrete functional space
containing the limit profile $\widetilde{V}$, and this is not
available in dimension $N=2$. However, if $N=2$, we can still provide
a rougher estimate of the eigenvalue variation.

\begin{theorem}\label{theor_asympotic_eigen_precise_u_0_dim2}
  Let $N=2$.  If $u_n^0(x_0)=0$ and $k\geq1$ is the vanishing order of
  $u_n^0$ at $x_0$, then
\begin{equation*}
\la_n^\e -\la_n^0 = O(\e^{2k-\delta}) \quad \text{as } \e \to 0^+, 
\end{equation*}
for every $\delta>0$.
\end{theorem}
From the blow-up analysis, it is also possible to obtain precise
information on the behavior of the perturbed eigenfunctions, as stated
in the following theorem.
\begin{theorem}\label{theor_asympotic_eigenfunction_precise_u_0_dim_big}
Let $N\geq 3$. Let 
\begin{align*}
    \Phi_\e(x):=\frac{u_n^\e(x_0+\e x)-u_n^0(x_0)}{\e^{h}} \quad \text{ and }
 \quad \Phi^0_\e(x):=\frac{u_n^0(x_0+\e x)-u_n^0(x_0)}{\e^{h}},
 \end{align*}
 where $u_n^0$ and $u_n^\e$ are as in \eqref{hp_u_n0} and
 \eqref{hp_u_ne} respectively, and $h$ is defined in
 \eqref{def_h}. Then, for every $R>0$ such that
 $\overline{\Sigma}\subset B_R$,
\begin{align}\label{eqn-1-rate-ef}
  \Phi_\e\rightarrow P_h^{u_n^0,x_0}-\widetilde{V} \quad \text{ in }
  H^1(B_R\setminus \overline{\Sigma}) \quad \text{ as } \e\rightarrow 0^+,
\end{align}
where $\widetilde{V}$ is the unique finite energy weak solution of
\eqref{srtel-rp} (in the sense clarified in \eqref{swrtel-rp}) and
$P_h^{u_n^0,x_0}$ is defined in \eqref{def_P_ui}. Furthermore,
\begin{align}\label{eqn-2-rate-ef}
    \lim_{\e\rightarrow 0} \e^{-(N+2h-2)}\|u_n^\e-u_n^0\|^2_{H^1(\Omega \setminus { \overline{\Sigma_\e}})}=\tr,
\end{align}
where $\tr$ is defined in \eqref{def_torsion_global}.
\end{theorem}
In dimension $N =2$, we have the following rougher result.
\begin{theorem}\label{theor_asympotic_eigenfunction_precise_u_0_dim_2}
  Let $N=2$. If 
$u_n^0$ and $u_n^\e$ are as in \eqref{hp_u_n0} and
 \eqref{hp_u_ne} respectively, and $h$ is defined in
 \eqref{def_h}, then  
\begin{equation*}
\|u_n^\e-u_n^0\|^2_{H^1(\Omega \setminus {\overline{\Sigma_\e}})}=O(\e^{2h-\delta}) \quad \text{as } \e \to 0^+, 
\end{equation*}
for any $\delta>0$.
\end{theorem}
In any dimension $N\geq2$, more explicit expansions of  the eigenvalue
variation can be derived for round holes,
i.e. for
\begin{equation}\label{eq:round-holes}
\Sigma= B_{r_1},
\end{equation}
for some $0<r_1<r_0$.

\begin{theorem}\label{theor_asympotic_eigen_precise_round_u_0_dim_big}
  Let $N\geq 3$ and $\Sigma$ be as in \eqref{eq:round-holes}.
  If $u_n^0(x_0)=0$ and $k\geq1$ is the vanishing order
  of $u_n^0$ at $x_0$, then
\begin{equation*}
  \la_n^\e -\la_n^0 = -\frac{k(N+2k-2)}{N+k-2}r_1^{N+2k-2}
  \left(\int_{\partial B_1} |P_k^{u_n^0,x_0}|^2\ds\right) \,
  \e^{N+2k-2}+o(\e^{N+2k-2})
  \quad \text{as } \e \to 0^+, 
\end{equation*}
where $P_k^{u_n^0,x_0}$ is defined in \eqref{def_P_ui}.
\end{theorem}
In dimension $2$, using  the arguments developed in \cite[Section
6]{FLO_Neumann}, we obtain the following expansion for round holes.
\begin{theorem}\label{theor_asympotic_eigen_precise_round_u_0_dim_2}
  Let $N=2$ and $\Sigma$ be as in \eqref{eq:round-holes}.
  If $u_n^0(x_0)=0$ and $k\geq1$ is the vanishing order
  of $u_n^0$ at $x_0$, then
  \begin{equation*}
   \la_n^\e-\la_n^0=-2k \pi r_1^{2k} \Bigg(\left|
     \pd{^ku_n^0}{x^k_1}(x_0)\right|^2+
   \frac{1}{k^2}\left|\pd{^ku_n^0}{x^{k-1}_1\partial
       x_2}(x_0)\right|^2 \Bigg) \e^{2k}+o( \e^{2k})
   \quad \text{as } \e \to 0^+.
\end{equation*}  
\end{theorem}

\begin{remark}
In this paper, we have assumed  $\alpha$ to be constant for clarity of
exposition.
However, our results can be generalized to the case where
$\alpha$ is a function over $\overline{\Omega}$.
More precisely, let $\alpha \in C^0(\overline{\Omega}\setminus \{x_0\})$
  be such that
\begin{equation*}
  \alpha_0:=\lim_{x \to x_0}\frac{\alpha(x)}{|x-x_0|^{\sigma}}
  \text{ exists, is finite and   }\alpha_0 \neq 0,
\end{equation*}
for some $\sigma \in (-1, +\infty)$.
For example, we expect that, if $u^n_0(x_0) \neq 0$, by following the proof of
Theorem~\ref{theor_asympotic_eigen_precise_u_not_0},
one could prove the expansion
\begin{equation*}
  \la_n^\e -\la_n^0 = \alpha_0 \mc{H}^{N-1}(\partial \Sigma )\, |u_n^0(x_0)|^2
  \e^{N+\sigma-1}+ o(\e^{N-1})
  \quad \text{as } \e \to 0^+;
\end{equation*}
similarly,  all the other results presented in
  Section~\ref{sec_main_results} are expected to be generalizable accordingly.
\end{remark}

\section{Preliminaries}\label{sec_prelimi}
\subsection{Functional setting}\label{subsection_func_sett}
For any open set $E\subset \R^N$, we denote with $L^2(E)$ and $H^1(E)$
the usual Lebesgue and Sobolev spaces, endowed with the scalar
products $(u,v)_{L^2(E)}:=\int_{E} uv \dx$ and
$(u,v)_{H^1(E)}:=\int_{E} \big(\nabla u \cdot \nabla v +uv\big)\dx$,
respectively.

The following proposition  provides a uniform extension result, due to
\cite{SW_extension}. For completeness, we provide a proof in
  Appendix \ref{sec_appendix-extension}.
\begin{proposition}\label{prop_extension}
  For $N \ge 2$, let $\Omega \subset \R^N$ and $\Sigma \subset \R^N$
  be open, bounded, and Lipschitz sets such that $\Omega$ is connected and
  \eqref{def_Sigma} is satisfied for some $x_0\in\Omega$ and
  $r_0>0$. For every $\e\in (0,1]$, let $\Sigma_\e$ be defined in
  \eqref{def_Sigma_e}.  Then, for every $\e \in (0,1]$ there exists an
  (inner) extension operator
\begin{equation*}
\Ee:H^1(\Omega \setminus \overline{\Sigma}_\e) \to H^1(\Omega) 
\end{equation*}
such that, for every
$u \in H^1(\Omega \setminus \overline{\Sigma}_\e)$ and $\e \in (0,1]$,
\begin{equation}\label{eq_Ee}
(\Ee u)\big|_{{\Omega \setminus \overline{\Sigma}_\e}}=u
\end{equation}
 and
\begin{equation}\label{ineq_extension}
\norm{\Ee u}_{H^1(\Omega)}^2\le K \|u\|_{H^1(\Omega \setminus \overline{\Sigma}_\e)}^2,
\end{equation}
for some constant $K>0$ which is independent of $\e$.
\end{proposition}

For clarity of exposition, in the following we will assume $x_0$ to be
the origin.  This is not a restrictive choice, as it can be easily verified
by translation. For every $\e\in (0,1]$ we define
\begin{equation}\label{eq:Omega-eps}
\Omega_\e := \Omega \setminus \overline{\Sigma}_\e.
\end{equation}

To study the variation of the eigenvalues as $\e\to 0^+$, it is important
to understand how the constants of the trace operators on the hole's
boundary depend on the parameter $\e$.
\begin{proposition}\label{prop_traces}
  Let $\Sigma_\e$ and $\Omega_\e$ be as in \eqref{def_Sigma_e} (with
  $x_0=0$) and \eqref{eq:Omega-eps}, respectively.
\begin{enumerate}[\rm (i)]
\item If $N \ge 3$, there exists a constant $C>0$, which does not
  depend on $\e$, such that, for every $\e \in (0,1]$ and
  $u \in H^1(\Omega_\e)$,
\begin{equation}\label{ineq_traces_N}
 \int_{ \partial \Si_\e} u^2 \ds \le C\e \int_{\Omega_\e}(|\nabla u|^2 +u^2) \dx.
\end{equation}  
\item If $N=2$, for every $\delta\in (0,1)$ there exists a constant
  $C_\delta>0$, depending on $\delta$ but independent of $\e$, such
  that, for every $\e \in (0,1]$ and $u \in H^1(\Omega_\e)$,
\begin{equation}\label{ineq_traces_2}
\int_{ \partial \Sigma_\e} u^2 \ds \le C_\delta \e^{1-\delta} \int_{\Omega_\e}(|\nabla u|^2 +u^2) \dx.
\end{equation}
\end{enumerate}
\end{proposition}

\begin{proof}
  Let $u \in H^1(\Omega_\e)$ and $u_\e(x):=(\Ee u)(\e x)$. Then
  $u_\e \in H^1(\Sigma)$. By the H\"older inequality, a change of
  variables, and classical Sobolev trace theory 
\begin{align*}
  \int_{ \partial \Si_\e} u^2 \ds &=\e^{N-1}\int_{ \partial \Si}
                                    u_\e^2 \ds
                                    \le C_1 \e^{N-1} \int_{\Si}(|\nabla u_\e|^2 +u_\e^2) \dx\\
                                  & = C_1 \e^{N-1}\int_{\Si}(\e^2|\nabla (\Ee u)(\e x)|^2 +|(\Ee u)(\e
                                    x)|^2) \dx\\
                                  &= C_1 \e^{-1}\int_{\Si_\e}(\e^2|\nabla (\Ee u)|^2 +|\Ee
                                    u|^2) \dx,
\end{align*}
for some positive constant $C_1>0$ independent of $\e$. Furthermore,
if $N\ge 3$, by the H\"older and Sobolev inequalities and
\eqref{ineq_extension}
\begin{equation*}
  \int_{\Sigma_\e} |\Ee u|^2 \dx \le |\Sigma|^{2/N}
  \e^{2}\left(\int_{\Omega} |\Ee u|^{\frac{2N}{N-2}}\dx\right)^{\!\!\frac{N-2}{N}}\\
\le C_2 \e ^{2} \|\Ee u\|_{H^1(\Omega)}^2 \le C_2 K \e ^{2}\|u\|_{H^1(\Omega_\e)}^2,
\end{equation*}
for some constant $C_2>0$ independent of $\e$. 
On the other hand, if $N=2$, for every  $\delta\in (0,1)$ we have 
\begin{equation*}
  \int_{\Sigma_\e} |\Ee u|^2 \dx \le |\Sigma|^{\frac{p-2}{p}}
  \e^{2-\delta}\left(\int_{\Omega} |\Ee u|^{p}\dx\right)^{\frac{2}{p}}\\
  \le C_3 \e ^{2-\delta} \|\Ee u\|_{H^1(\Omega)}^2\le C_3 K \e
  ^{2-\delta} \|u\|_{H^1(\Omega_\e)}^2,
\end{equation*}
where $p=\frac{4}{\delta}$ and $C_3>0$ is a positive constant
depending on $\delta$. Hence \eqref{ineq_traces_N} and
\eqref{ineq_traces_2} are proved.
\end{proof}

\begin{remark}\label{remark_traces}
  By the classical Sobolev trace inequality on the domain
  $\Omega \setminus \overline{B_{r_0}}$, there exists a constant
  $C_{r_0}>0$,  depending on $\Omega$ and $r_0$, such that
\begin{equation*}
  \left(\int_{\partial \Omega} v^2 \, \ds\right)^{\frac{1}{2}} \le
  C_{r_0}\norm{v}_{H^1(\Omega \setminus \overline{B_{r_0}})} \le C_{r_0} \norm{v}_{H^1(\Omega_\e)}
\end{equation*}
for every $\e\in(0,1]$ and $v \in H^1(\Omega_\e)$.
\end{remark}

For every  $\alpha\in\R$ and $c\ge 0$, we consider  the eigenvalue problem 
\begin{equation}\label{prob_robin_moved}
\begin{cases}
-\Delta u+c u= \mu u, & \text{ in } \Omega,\\
\pd{u}{\nu} + \alpha u=0, & \text{ on } \partial \Omega,
\end{cases}
\end{equation}
which is equivalent to \eqref{prob_robin} in the sense that it has the
same eigenfunctions and eigenvalues that are shifted by an amount $c$.
For every $\e \in (0,1]$, the corresponding perturbed problem is
\begin{equation}\label{prob_robin_perturbed_moved}
\begin{cases}
-\Delta u+cu= \mu u, & \text{ in } \Omega_\e,\\
\pd{u}{\nu} + \alpha u=0, & \text{ on } \partial \Omega,\\
\pd{u}{\nu} + \alpha u=0, & \text{ on } \partial \Sigma_\e,
\end{cases}
\end{equation}
where $\nu$ is the unit outer normal vector of $\Omega_\e$ on $\partial \Omega_\e$.
For every $\alpha\in\R$, $c\ge 0$, and $\e \in [0,1] $, we define the
symmetric bilinear form
\begin{equation}\label{eq:e-eca}
  q_{\e,c,\alpha}:H^1(\Omega_\e) \times H^1(\Omega_\e) \to \R, \quad 
  q_{\e,c,\alpha}(v,w)= \int_{\Omega_\e} \left(\nabla  v\cdot \nabla w
    +c vw\right)
  \dx +\alpha \int_{\partial \Omega_\e} vw \ds,
\end{equation}
where $\Omega_0:=\Omega$. 
We say that $u\in H^1(\Omega_\e)$ is an eigenfunction of
problem \eqref{prob_robin_perturbed_moved} associated to the
eigenvalue $\mu\in\R$ if 
\begin{equation*}
  q_{\e,c,\alpha}(u,v)=\mu \int_{\Omega_\e} uv \dx\quad
  \text{for every }v\in H^1(\Omega_\e),
\end{equation*}
while $u\in H^1(\Omega)$ is a eigenfunction of problem
\eqref{prob_robin_moved} associated to the eigenvalue $\mu$ if
\begin{equation*}
q_{0,c,\alpha}(u,v)=\mu \int_{\Omega} uv \dx \quad\text{for every }v \in H^1(\Omega).
\end{equation*}
By \cite[Theorem 18.1]{L_book_sobolev} the bilinear form
$q_{\e,c,\alpha}$ is semibounded from below. By classical
spectral theory, the spectrum of problem \eqref{prob_robin_perturbed_moved} is a
diverging sequence
$\{\mu_{j,\e,c, \alpha}\}_{j \in \mathbb{N}\setminus\{0\}}$.  It is
clear that
  \begin{equation*}
\mu_{j,\e,c,\alpha}=
c+\lambda_j^\e\quad \text{for every } c\in\R \text{ and }j\in\N\setminus\{0\},
\end{equation*}
where $\{\la_j^\e\}_{j \in \mathbb{N}\setminus\{0\}}$ are the
eigenvalues of \eqref{prob_robin_perturbed}.
Hence, studying the eigenvalue variation for problems
\eqref{prob_robin} and \eqref{prob_robin_perturbed} is equivalent to
studying the eigenvalue variation for problems
\eqref{prob_robin_moved} and \eqref{prob_robin_perturbed_moved} for
any $c\in\R$.  Furthermore, by \cite[Theorem 18.1]{L_book_sobolev} and
Proposition \ref{prop_traces}, for any $\alpha \in \R$ there exist
$c_\alpha>0$, $C_\alpha>0$, and $\e_\alpha\in(0,1)$, such that, for
all $\e \in [0,\e_\alpha)$,
\begin{equation*}
\mu_{1,\e,c_\alpha, \alpha}>1
\end{equation*}
and 
\begin{equation}\label{ineq_norm_qe}
\norm{v}_{H^1(\Omega_\e)}^2 \le C_\alpha q_{\e,c_\alpha ,\alpha} (v,v)
\quad
\text {for every } v\in H^1(\Omega_\e).
\end{equation}
In view of Remark \ref{remark_traces} and
\eqref{ineq_traces_N}--\eqref{ineq_traces_2}, there exists a constant
$\kappa_\alpha>0$ which depends only on $\Omega$, $r_0$, $\Sigma$,
and $\alpha$ such that
\begin{equation}\label{ineq_qe_norm}
\norm{v}_{H^1(\Omega_\e)}^2 \ge \kappa_\alpha q_{\e,c_\alpha ,\alpha}
(v,v) \quad
\text {for every } v\in H^1(\Omega_\e).
\end{equation}
We will study the eigenvalue variation for problems
\eqref{prob_robin_moved} and \eqref{prob_robin_perturbed_moved} with
$c=c_\alpha$.  For any $\alpha \in \R$ fixed (and $c_\alpha$ chosen
accordingly as explained above), let us define
\begin{equation}\label{def_qe}
q_\e:=q_{\e,c_\alpha,\alpha}\quad  \text{ for every }  \e \in [0,1],
\end{equation}
and 
\begin{equation}\label{def_la_ne}
  \mu_{j}^\e:=
  \mu_{j,\e,c_\alpha,\alpha}\quad\text{for every }
  j\in \mathbb{N}\setminus\{0\}\text{ and }\e \in [0,1].
\end{equation}
By classical spectral theory, there exists an orthonormal basis
$\{u_j^\e\}_{j \in \mb{N}\setminus\{0\}}$ of $L^2(\Omega_\e)$ such that $u_{j}^\e$ is
an eigenfunction associated to the eigenvalue
$\mu_{j}^\e$; consequently, it
  is also an eigenfunction of \eqref{prob_robin_perturbed} associated
  to
  \begin{equation}\label{eq:rel-aut}
  \lambda_j^\e=\mu_j^\e-c_\alpha.
\end{equation}
If $j=n$,
  $u_n^0$ and $u_n^\e$ are chosen to satisfy \eqref{hp_u_n0} and
  \eqref{hp_u_ne} respectively.

\subsection{Spectral stability}\label{subsection_spec_stab}
In this subsection, we prove that $\mu_j^\e \to \mu_j^0$ (and hence
$\la_j^\e \to \la_j^0$) as $\e \to 0^+$ for every
$j \in \mb{N}\setminus\{0\}$.  To this end, for every
$\e \in (0,\e_\alpha)$, let $R_\e$ be the inverse of the operator
$-\Delta +c_\alpha$ on $\Omega_\e$ under Robin boundary conditions.
More precisely, for every $\e \in (0,\e_\alpha)$,
\begin{equation*}
R_\e :L^2(\Omega_\e) \to H^1(\Omega_\e)\subset L^2(\Omega_\e),  
\end{equation*}
and, for every $f \in L^2(\Omega_\e)$, $R_\e f $ is the unique weak
solution to the problem
\begin{equation*}
\begin{cases}
-\Delta u+c_\alpha u= f, & \text{ in } \Omega_\e,\\
\pd{u}{\nu} + \alpha u=0, & \text{ on } \partial \Omega,\\
\pd{u}{\nu} + \alpha u=0, & \text{ on } \partial \Sigma_\e,
\end{cases}
\end{equation*}
i.e.
  \begin{equation}\label{eq:deboleris}
    q_\e(R_\e f,v)=\int_{\Omega_\e}fv\dx\quad\text{for all }
    v\in H^1(\Omega_\e).
  \end{equation}
  Similarly, we can define $R_0$ as the inverse of the operator
  $-\Delta +c_\alpha$ on $\Omega$ under Robin boundary conditions.
  For every $\e \in [0,\e_\alpha)$, the operator $R_\e$ is linear and
  bounded as an operator from $L^2(\Omega_\e)$ into $H^1(\Omega_\e)$;
  furthermore, $R_\e$, understood as an operator from $L^2(\Omega_\e)$
  into $L^2(\Omega_\e)$, is compact and self-adjoint. Moreover, since
  $\mu_1^\e>1$, we have $\norm{R_\e}_{\mc{L}(L^2(\Omega_\e))}\le1$,
  where $\norm{R_\e}_{\mc{L}(L^2(\Omega_\e))}$ is the operator norm.

For every $\e \in (0,1]$, we define the linear and continuous operators
\begin{equation*}
P_\e:L^2(\Omega) \to L^2(\Omega_\e), \quad P_\e u:=u_{|\Omega_\e}, 
\end{equation*}
and 
\begin{equation*}
\widetilde{R}_\e:L^2(\Omega) \to L^2(\Omega), \quad \widetilde{R}_\e u:=\Ee R_\e P_\e u.
\end{equation*}

\begin{proposition}\label{prop_Re_to_R0}
For every $u \in L^2(\Omega)$
\begin{equation*}
\lim_{\e \to 0^+}\big\|\widetilde{R}_\e u -R_0 u\big\|_{L^2(\Omega)} =0.    
\end{equation*}
\end{proposition}

\begin{proof}
  Let $u \in L^2(\Omega)$. Then by \eqref{ineq_extension},
  \eqref{ineq_norm_qe}, and \eqref{eq:deboleris}, and the fact that
  $\norm{R_\e}_{\mc{L}(L^2(\Omega_\e))} \le 1$,
\begin{align}
\notag \|\widetilde{R}_\e u\|_{H^1(\Omega)}^2 
&\leq K\|R_ \e P_\e u\|_{H^1(\Omega_\e)}^2
\le K  C_\alpha q_\e( R_\e  P_\e u,R_\e  P_\e u)\\ &
\label{ineq_tilde_Re_unifor_bound}=  K C_\alpha (P_\e u,R_\e  P_\e u)_{L^2(\Omega_\e)}  \le K 
C_\alpha\norm{u}^2_{L^2(\Omega)}
\end{align}
for every $\e\in (0,\e_\alpha)$.  In particular, the family
$\{ \widetilde{R}_\e u\}_{\e \in (0,\e_\alpha)}$ is bounded in
$H^1(\Omega)$. Hence there exist $v \in H^1(\Omega)$ and a sequence
$\e_j \to 0^+$ such that $\widetilde{R}_{\e_j} u \rightharpoonup v $
weakly in $H^1(\Omega)$ as $j \to \infty$. Furthermore, for every
$w \in H^1(\Omega)$ and $j\in\N$ we have
\begin{equation}\label{proof_Re_to_R0:1}
q_{\e_j}(R_{\e_j} P_{\e_j} u, w)=(P_{\e_j}u, w)_{L^2(\Omega_{\e_j})}.
\end{equation}
We would like to pass to the limit as $j \to \infty$ in \eqref{proof_Re_to_R0:1}.
By \eqref{def_qe}
\begin{align*}
  q_{\e_j}(R_{\e_j} P_{\e_j} u,   w)&= \int_{\Omega_{\e_j}} \big(\nabla
  (R_{\e_j}
  P_{\e_j} u)\cdot \nabla  w 
  +c_\alpha (R_{\e_j} P_{\e_j} u)  w \big)\dx +\alpha \int_{\partial \Omega_{\e_j}}
  (R_{\e_j} P_{\e_j} u ) w \ds\\
  &=\int_{\Omega}\big( \nabla  (\widetilde{R}_{\e_j} u)\cdot  \nabla w 
    +c_\alpha(\widetilde{R}_{\e_j} u)  w \big)\dx +\alpha \int_{\partial \Omega} (R_{\e_j} P_{\e_j} u ) w \ds\\
  &\qquad-\int_{\Si_{\e_j}}
  \big(\nabla  (\widetilde{R}_{\e_j} u)\cdot \nabla  w 
  +c_\alpha(\widetilde{R}_{\e_j} u) w \big)\dx 
  +\alpha \int_{\partial \Sigma_{\e_j}} (R_{\e_j} P_{\e_j} u ) w \ds
\end{align*}
and, by the Cauchy-Schwarz inequality,
\begin{equation*}
  \left|\int_{\Si_\e} \big(\nabla  (\widetilde{R}_\e u)\cdot \nabla w
    +c_\alpha   (\widetilde{R}_\e u) w\big)\dx \right|\le
  \max\{1,c_\alpha\}
  \|\widetilde{R}_\e u\|_{H^1(\Omega)} \|w\|_{H^1(\Sigma_\e)} \to 0
\end{equation*}
as $\e \to 0^+$,
thanks to the absolute continuity of integrals, since
$|\Si_{\e}|=\e^{N}|\Si| \to 0^+$.  Hence, thanks to
\eqref{ineq_traces_2} if $N=2$ or \eqref{ineq_traces_N} if $N \ge 3$,
we may pass to the limit as $j\to \infty$ in
\eqref{proof_Re_to_R0:1}, thus obtaining
\begin{equation*}
q_{0}(v,w)=(u,w)_{L^2(\Omega)} \quad \text{for all } w \in H^1(\Omega).
\end{equation*}
Hence $v=R_0u$. By the Urysohn Subsequence principle and the
Rellich-Kondrachov Theorem, we conclude that
$\widetilde{R}_\e u \to R_0 u$ strongly in $L^2(\Omega)$.
\end{proof}
Based on Proposition \ref{prop_Re_to_R0}, we are now able to prove the
stability result stated in Theorem \ref{theorem_spectral_stability}.
\begin{proof}[Proof of Theorem \ref{theorem_spectral_stability}]
 We claim that 
\begin{equation} \label{limit_Re_Ro_oper_norm}
\lim_{\e \to 0^+}\|\widetilde{R}_\e  -R_0\|_{\mc{L}(L^2(\Omega))} =0.    
\end{equation}
By compactness of the operator $\widetilde{R}_\e -R_0$, for every
$\e \in (0,\e_\alpha)$ there exists some $f_\e \in L^2(\Omega)$ with
$\norm{f_\e}_{L^2(\Omega)}=1$ such that
\begin{equation*}
  \|R_0-\widetilde{R}_\e\|_{\mc{L}(L^2(\Omega))}=
  \|R_0 f_\e-\widetilde{R}_\e f_\e\|_{L^2(\Omega)}.
\end{equation*}
Since $\{f_\e\}_\e$ is bounded in $L^2(\Omega)$, there exist a
sequence $\e_j \to 0^+$ and some $f \in L^2(\Omega)$ such that
$f_{\e_j} \rightharpoonup f$ weakly in $L^2(\Omega)$ as $j\to
\infty$. By compactness of the operator $R_0$,
$R_0 f_{\e_j} \to R_0 f$ strongly in $L^2(\Omega)$, as $j \to \infty$.
Furthermore, thanks to \eqref{ineq_tilde_Re_unifor_bound}, up to a
further subsequence, $\widetilde{R}_{\e_j}f_{\e_j} \rightharpoonup g$
weakly in $H^1(\Omega)$ for some $g\in H^1(\Omega)$.  For every
$v \in L^2(\Omega)$, since $R_\e$ is self-adjoint we have
\begin{equation*}
  (\widetilde{R}_{\e_j}f_{\e_j},v)_{L^2(\Omega)}=
(  P_{\e_j} f_{\e_j},R_{\e_j}P_{\e_j}v)_{L^2(\Omega_{\e_j})}+
  (\widetilde{R}_{\e_j}f_{\e_j},v)_{L^2(\Si_{\e_j})}. 
\end{equation*}
By \eqref{ineq_tilde_Re_unifor_bound} and the absolute continuity of
integrals, we have
\begin{equation*}
  (\widetilde{R}_{\e_j}f_{\e_j},v)_{L^2(\Si_{\e_j})} \to 0^+ \quad
  \text{as } j \to \infty.
\end{equation*}
Furthermore,
\begin{equation*}
  (P_{\e_j} f_{\e_j},R_{\e_j}P_{\e_j}v)_{L^2(\Omega_{\e_j})}=
  (f_{\e_j},\widetilde{R}_{\e_j}v)_{L^2(\Omega)} -
  (f_{\e_j},\widetilde{R}_{\e_j}v)_{L^2(\Si_{\e_j})}
\end{equation*}
and so, by Proposition \ref{prop_Re_to_R0} and the absolute continuity of
integrals,
\begin{equation*}
(g,v)_{L^2(\Omega)}=\lim_{k\to \infty}
  (\widetilde{R}_{\e_j}f_{\e_j},v)_{L^2(\Omega)}=(f,R_0 v)_{L^2(\Omega)}
  =(R_0 f,v)_{L^2(\Omega)}.
\end{equation*}
Hence $g=R_0 f$. We conclude that
$\widetilde{R}_{\e_j}f_{\e_j} \to R_0f$ strongly in $L^2(\Omega)$ as
$j\to \infty$.  By the Urysohn Subsequence principle and compactness
of $R_0$, we deduce \eqref{limit_Re_Ro_oper_norm}.

It is easy to verify that
$\sigma( \widetilde{R}_\e)= \big\{\frac1{\mu_j^\e}\big\}_{j \in
  \mathbb{N}\setminus\{0\}} \cup \{0\}$, where we are denoting with
$\sigma( \widetilde{R}_\e)$ the spectrum of $\widetilde{R}_\e$.  Hence
\eqref{eq_limit_eigen} follows from \eqref{limit_Re_Ro_oper_norm}, 
\cite[Chapter XI 9.5, Lemma 5, Page~1091]{DSJ_book}, and \eqref{eq:rel-aut}.
\end{proof}

\section{Eigenvalue variation}\label{sec_eign_var}

For any $\alpha \in \R$ fixed, let $c_\alpha>0$ and
$\e_\alpha\in(0,1)$ be as in Section \ref{sec_prelimi}. For every
$\e\in(0,\e_\alpha)$, we consider the bilinear form
$q_\e=q_{\e,c_\alpha,\alpha}$ defined in
\eqref{eq:e-eca} and \eqref{def_qe}.  Let
$n \in \mathbb{N}\setminus \{0\}$ be such that \eqref{eq:simple} is
satisfied; hence, by \eqref{eq:rel-aut}, also the eigenvalue
$\mu_n^0$, defined as in \eqref{def_la_ne}, is simple. Let $u_n^0$ be
an associated eigenfunction satisfying \eqref{hp_u_n0}.

For every $\e \in (0, \e_\alpha)$, we define the linear functional
$L_\e : H^1(\Omega_\e)\to\R$ as
\begin{align}
  \label{def_Le}L_\e(v):=q_\e(u_n^0 ,v)-\mu_n^0 \int_{\Omega_\e}  u^0_n v \dx =
    \int_{\partial \Sigma_\e} \left(\pd{u_n^0}{\nu} +\alpha u_n^0 \right) v \ds,
\end{align}
where $\nu$ is the outer normal unit vector to $\Omega_\e$ on
$\partial \Sigma_\e$. We also consider the quadratic functional
\begin{equation*}
J_\e: H^1(\Omega_\e)\to\R,\quad 
  J_\e(v):=\frac{1}{2}\,q_\e(v,v)-L_\e(v).
\end{equation*}
We omit the proof of the following proposition, since it is
standard. We refer to \cite[Proposition 3.4]{FLO_Neumann} and
\cite[Proposition 2.4]{BS_quantitative} for analogous results in a
different setting.
\begin{proposition}\label{propo_tau_domain}
  For every $\e\in (0,\e_\alpha)$, there exists a unique function
  $V_\e \in H^1(\Omega_\e)$ such that
\begin{equation*}
J_\e(V_\e)=\min\{J_\e(v): v \in H^1(\Omega_\e)\}.
\end{equation*}
Furthermore, $V_\e$ is the unique weak solution to the problem
\begin{equation*}
\begin{dcases}
-\Delta V_\e +c_\alpha V_\e = 0, & \text{in } \Omega_\e, \\
\partial_\nu V_\e  +    \alpha V_\e =0, & \text{on } \partial\Omega,\\
\partial_\nu V_\e  +    \alpha V_\e =\partial_\nu u_n^0  +    \alpha u_n^0, & \text{on } \partial\Sigma_\e,\\
\end{dcases}
\end{equation*}
that is,
\begin{equation}\label{eq_Ve}
q_\e(V_\e,v)=L_\e(v) \quad \text{for every } v \in  H^1(\Omega_\e).
\end{equation}
\end{proposition}
For every $\e\in (0,\e_\alpha)$, we consider the quantity 
\begin{equation}\label{def_Te}
\mc{T}_\e:=-2\left(\frac{1}{2}\,q_\e(V_\e,V_\e)-L_\e(V_\e)\right)=q_\e(V_\e,V_\e).
\end{equation}
The quantity $\mc{T}_\e$, which we refer to as \textit{torsional
  rigidity}, turns out to be an important ingredient in the asymptotic
expansion of the eigenvalue variation.  An alternative
characterization of $\mc{T}_\e$ is given in the following lemma, whose
proof is omitted, as it is similar to the ones given in
\cite[Proposition 6.5]{FNOS_AB_12} and \cite[Lemma 3.1]{FLO_Neumann}
in analogous situations.
\begin{lemma}\label{lemma_torsion_equiv}
For every $\e \in (0,\e_\alpha)$,
\begin{equation}\label{eq_chara_torsion}
\mc{T}_\e=\sup_{u \in H^1(\Omega_\e) \setminus \{0\}} \frac{(L_\e(u))^2}
{q_\e(u,u)}.
\end{equation}
\end{lemma}
Thanks to the above lemma, we obtain the following quantitative
estimate of the vanishing order of $\mc{T}_\e$ as $\e \to 0^+$.
\begin{proposition}\label{prop_estimates_torsion}
  Let  $k\geq 0$ be the vanishing order of $u_n^0$ at $x_0=0$ and
    $h:= \max\{k,1\}$. Then 
\begin{equation}\label{eq_estimate_tors_N}
\mc{T}_\e=
\begin{cases}
O(\e^{N+2h-2}), & \text{ if } N\geq 3,\\
O(\e^{2h-\delta}), & \text{ if } N=2,
\end{cases}\quad\text{as }\e\to0^+,
\end{equation}
for every $\delta\in(0,1)$.
In particular,
\begin{equation}\label{limit_Ve_norm_0}
  \lim_{\e \to 0^+}q_\e(V_\e,V_\e)=\lim_{\e \to 0^+}
  \norm{V_\e}^2_{H^1(\Omega_\e)}=\lim_{\e \to 0^+}\mc{T}_\e=0.
\end{equation}
\end{proposition}
\begin{proof}
If $N \ge 3$, by \eqref{ineq_traces_N} and \eqref{ineq_norm_qe}, 
\begin{align*}
  \int_{\partial \Sigma_\e}\left[ \pd{u_n^0}{\nu}+\alpha u_n^0\right]u
  \ds&
       \le \sqrt{C}\, \e^{\frac{1}{2}}\norm{\pd{u_n^0}{\nu}+
       \alpha u_n^0}_{L^2(\partial \Sigma_\e)}\norm{u}_{H^1(\Omega_\e)} \\
     &\le \sqrt{C}\, \sqrt{C_\alpha} \e^{\frac{1}{2}}\norm{\pd{u_n^0}{\nu}+\alpha
       u_n^0}_{L^2(\partial \Sigma_\e)} \sqrt{q_\e(u,u)}
\end{align*}
for every $u \in H^1(\Omega_\e)$.
Hence, in view of
\eqref{eq_chara_torsion},
\begin{equation}\label{eq:first-est-Te}
  \mc{T}_\e \le  C C_\alpha \e \int_{\partial \Sigma_\e}
  \bigg(\pd{u_n^0}{\nu}+\alpha u_n^0\bigg)^2 \ds.
\end{equation}
If $u_n^0(0)=0$, then $k=h\geq1$, $u_n^0 (x)=O(|x|^k)$ and 
$|\nabla u_n^0 (x)|=O(|x|^{k-1})$ as $x\to0$, so that
\eqref{eq:first-est-Te} yields $\mathcal T_\e=O(\e^{N+2k-2})=O(\e^{N+2h-2})$ as
$\e\to0^+$. If $u_n^0(0)\neq 0$, then $h=1$ and, by boundedness of
$u_n^0$ and $\nabla u_n^0$ in a neighbourhood of $0$,
\eqref{eq:first-est-Te}
directly implies that  $\mathcal T_\e=O(\e^{N})=O(\e^{N+2h-2})$ as
$\e\to0^+$. Estimate \eqref{eq_estimate_tors_N} is thereby proved in
the case $N\geq3$.

The proof of \eqref{eq_estimate_tors_N} in the case $N=2$ proceeds
similarly, using \eqref{ineq_traces_2} instead of \eqref{ineq_traces_N}.
\end{proof}

\begin{proposition}\label{prop_norm_Ve_L2_H1}
Let $V_\e$ be as in Proposition \ref{propo_tau_domain}. Then 
\begin{equation}\label{eq:norm_Ve1}
\norm{V_\e}^2_{L^2(\Omega_\e)}=o\left(\mc{T}_\e\right) \quad  \text{as } \e \to 0^+,
\end{equation}
and 
\begin{equation}\label{eq:norm_Ve2}
\norm{V_\e}^2_{L^2(\partial \Omega)}=o\left(\mc{T}_\e\right) \quad  \text{as } \e \to 0^+.
\end{equation}
\end{proposition}
\begin{proof}
In view of \eqref{ineq_norm_qe}, to prove \eqref{eq:norm_Ve1} it is enough to show that 
\begin{equation*}
  \|V_\e\|^2_{L^2(\Omega_\e)}=o\big(\|V_\e\|^2_{H^1(\Omega_\e)}\big)
\quad  \text{as } \e \to 0^+.
\end{equation*}
We argue by contradiction. Suppose that there exist a positive
constant $c>0$ and a sequence $\e_j\to0^+$ such that
\begin{equation*}
  \|V_{\e_j}\|^2_{L^2(\Omega_{\e_j})} \ge c
  \|V_{\e_j}\|^2_{H^1(\Omega_{\e_j})}
  \quad \text{for every } j \in \mb{N}.
\end{equation*}
For every $j\in\N$, we define
\begin{equation*}
  W_{j}:=\frac{
    \mathsf{E}_{\e_j}(V_{\e_j})}{\|
\mathsf{E}_{\e_j}
      (V_{\e_j}) \|_{L^2(\Omega)}},
\end{equation*}
where $\mathsf{E}_{\e_j}$ is the extension operator introduced
in Proposition \ref{prop_extension}. Then $\{W_{j}\}_{j \in \mb{N}}$
is bounded in $H^1(\Omega)$ by Proposition \ref{prop_extension}. It
follows that, up to a subsequence, $W_{j} \rightharpoonup W$ weakly in
$H^1(\Omega)$ as $j\to+ \infty$, for some $W \in H^1(\Omega)$.  The
compactness of the embedding $H^1(\Omega)\subset L^2(\Omega)$ implies
that $\|W\|_{L^2(\Omega)}=1$.

Let $\varphi \in C^\infty_c(\overline \Omega\setminus \{0\})$. Then
there exists $j_0 \in \mb{N}$ such that
$\varphi \in C^\infty_c(\overline \Omega\setminus\overline{\Sigma_{\e_j}})$ for
all $j \ge j_0$. Then, by \eqref{eq_Ee} and \eqref{eq_Ve},
\begin{equation}\label{eq:Wj}
  \int_{\Omega} \big(\nabla W_{j} \cdot \nabla \varphi
  +c_\alpha  W_{j}  \varphi \big) \dx +\alpha  \int_{\partial\Omega}   W_{j}  \varphi \ds=0.
\end{equation}
Passing to the limit as $j \to +\infty$, we conclude that 
\begin{equation}\label{eq:W}
  \int_{\Omega}\big(\nabla W\cdot \nabla \varphi
  +c_\alpha  W  \varphi \big)\dx +\alpha  \int_{\partial\Omega}   W
  \varphi\ds=0
  \quad\text{for every }\varphi\in C^\infty_c(\overline \Omega\setminus \{0\}).
\end{equation}
Since $C^\infty_c(\overline \Omega\setminus \{0\})$ is dense in
$H^1(\Omega)$, we may test the above equation with $W$ and conclude
that $W=0$ in view of \eqref{ineq_norm_qe} with $\e=0$.  This is a
contradiction because $\|W\|_{L^2(\Omega)}=1$. Estimate
\eqref{eq:norm_Ve1} is thereby proved.

To prove \eqref{eq:norm_Ve2} we argue in a similar way, and
  assume, by contradiction, that there exist $c>0$ and a sequence $\e_j\to0^+$ such that
\begin{equation*}
  \|V_{\e_j}\|^2_{L^2(\partial \Omega)} \ge c
  \|V_{\e_j}\|^2_{H^1(\Omega_{\e_j})}
  \quad \text{for every } j \in \mb{N}.
\end{equation*}
For every $j\in\N$, we define
$Z_{j}:=\|V_{\e_j} \|_{L^2(\partial \Omega)}^{-1}
\mathsf{E}_{\e_j}(V_{\e_j})$. By Proposition \ref{prop_extension},
it follows that $\{Z_{j}\}_{j \in \mb{N}}$ is bounded in $H^1(\Omega)$,
hence there exists $Z \in H^1(\Omega)$ such that
$Z_{j} \rightharpoonup Z$ weakly in $H^1(\Omega)$, up to a
subsequence.  The compactness of the trace map 
$H^1(\Omega)\to L^2(\partial \Omega)$ implies that
$\|Z\|_{L^2(\partial \Omega)}=1$.

Letting $\varphi \in C^\infty_c(\overline \Omega\setminus \{0\})$, we
have that $Z_j$ satisfies \eqref{eq:Wj} for all $j$ sufficiently
large, so that, passing to the limit, $Z$ satisfies \eqref{eq:W} for
every $\varphi \in C^\infty_c(\overline \Omega\setminus \{0\})$, and
hence, by density of $C^\infty_c(\overline \Omega\setminus \{0\})$ in
$H^1(\Omega)$, for every $\varphi \in H^1(\Omega)$. As above, this
would imply that $Z=0$ in view of \eqref{ineq_norm_qe}, giving rise to
a contradiction, because $\|Z\|_{L^2(\partial \Omega)}=1$. Estimate
\eqref{eq:norm_Ve2} is thereby proved.
\end{proof}

The following theorem applies a known perturbative theory for
eigenvalue problems in perforated domains to our context.  Originally
developed by Courtois in \cite{C_mainfols_holes} for the Dirichlet
problem, this theory is based on the \emph{Lemma on small eigenvalues}
by Y. Colin de Verdi\`ere \cite{ColindeV1986}. It was later revisited
in \cite{AFHC_spec_AB} for simple eigenvalues and in
\cite{ALM_cap_exp} for multiple ones, whereas a recent extension to
Neumann problem is due to \cite{FLO_Neumann}.  A generalization to a
broad class of abstract perturbative problems -- including the
aforementioned ones -- is provided in \cite{BS_quantitative}; a
simplified version of this, adapted to our case, is presented in
Appendix \ref{sec_appendix} and will now be applied to problem
\eqref{prob_robin_perturbed}.
  \begin{theorem}\label{theo_eigen_var_abs}
  Let $n\in\N\setminus\{0\}$ be such that the $n$-th eigenvalue
  $\lambda_n^0$ of \eqref{prob_robin} satisfies \eqref{eq:simple},
  with an associated eigenfunction $u_n^0$ as in \eqref{hp_u_n0};
  for every $\e\in (0,1)$, let $\lambda_n^\e$ be the $n$-th eigenvalue
  of \eqref{prob_robin_perturbed}. Then
  \begin{align}
 \notag \la_n^\e -\la_n^0&= -\mc{T}_\e -
  \int_{\Sigma_\e}\big( |\nabla u_n^0|^2-\la_n^0  |u_n^0|^2\big)
  \dx
  +\alpha\int_{\partial \Sigma_\e}  |u_n^0|^2\ds\\
    \label{eq_eigen_var_abs}
    &\quad+o(\mc{T}_\e)+o\left(\int_{\Sigma_\e}\big( |\nabla u_n^0|^2
    -\la_n^0 |u_n^0|^2\big) \dx\right)
  +o\left(\int_{\partial \Sigma_\e}  |u_n^0|^2\ds\right)
  \end{align}
  as $\e \to 0^+$.
Moreover, if  $u_n^\e$ is  the eigenfunction of
\eqref{prob_robin_perturbed} associated with the eigenvalue $\la_n^\e$
chosen as in \eqref{hp_u_ne} for every $\e\in(0,1)$, then 
\begin{align}
&\|u_n^\e - (u_n^0
                -V_\e)\|^2_{H^1(\Omega_\e)}=o\left(\mc{T}_\e\right)
                +O\big(\|u_n^0\|_{L^2(\Sigma_\e)}^4\big), \label{eq_eigenfun_abs_with_Ve}\\
  &\|u_n^\e - u_n^0\|^2_{H^1(\Omega_\e)}=\mc{T}_\e
    +o\left(\mc{T}_\e\right)+O\big(\|u_n^0\|_{L^2(\Sigma_\e)}^4\big)
    +O\left(\mc{T}_\e^{1/2}\|u_n^0\|_{L^2(\Sigma_\e)}^2\right)
    , \label{eq_eigenfun_abs}
\end{align}
as $\e\to0^+$.
\end{theorem}
\begin{proof}
In view of \eqref{def_Sigma_e},
\begin{equation}\label{proof_eigen_var_abs_1}
\lim_{\e \to 0^+} |\overline{\Sigma_\e}|_N=0,
\end{equation}
where $|\overline{\Sigma_\e}|_N$ denotes the $N$-dimensional
  Lebesgue measure of $\overline{\Sigma_\e}$.  This, together with
\eqref{eq:simple}, Theorem \ref{theorem_spectral_stability},
\eqref{eq:rel-aut}, \eqref{ineq_norm_qe}, and \eqref{limit_Ve_norm_0},
ensures that the assumptions of Theorem \ref{theorem_expansion_very_abs} are
satisfied.  Hence, Theorem
\ref{theorem_expansion_very_abs} yields
\begin{equation}\label{proof_eigen_var_abs_2}
  \la_n^\e -\la_n^0=\mu_n^\e -\mu_n^0=
  \frac{\mu_n^0\int_{\Omega_\e} u_n^0 V_\e \dx
    +O\big(\|V_\e\|^2_{L^2(\Omega_\e)}
    \big)}{\int_{\Omega_\e} |u_n^0|^2 \dx+O\big(\|V_\e\|_{L^2(\Omega_\e)}\big)} 
  \quad \text{as }\e \to 0^+.
\end{equation}
To clarify the above expansion further, we first observe, by
\eqref{hp_u_n0}, \eqref{proof_eigen_var_abs_1}, and \eqref{limit_Ve_norm_0},  that
\begin{equation}\label{proof_eigen_var_abs_3}
  \frac{1}{\int_{\Omega_\e} |u_n^0|^2
    \dx+O(\norm{V_\e}_{L^2(\Omega_\e)}) }
  =\frac{1}{1+o(1)} = 1+o(1)\quad \text{as }\e \to 0^+.
\end{equation}
Furthermore, 
\begin{equation}\label{proof_eigen_var_abs_4}
O(\|V_\e\|_{L^2(\Omega_\e)}^2)= o(\mc{T}_\e) \quad \text{as }\e \to 0^+,
\end{equation}
in view of Proposition \ref{prop_norm_Ve_L2_H1}.  Finally, by
\eqref{def_Le} tested with $V_\e$, \eqref{eq_Ve}
tested with $V_\e$, \eqref{eq_Ve} tested with $u_n^0$, 
\eqref{prob_robin}, 
and an integration by part,
\begin{align}
 \notag \mu_n^0\int_{\Omega_\e} u_n^0 V_\e \dx&=q_\e(V_\e,u_n^0)
                                                -L_\e(V_\e)=L_\e(u_n^0)-\mc{T}_\e\\
 \notag &= \int_{\partial \Sigma_\e} \pd{u_n^0}{\nu} u_n^0 \ds 
+\alpha \int_{\partial \Sigma_\e}  |u_n^0|^2 \ds-\mc{T}_\e\\
\label{proof_eigen_var_abs_5}&= -\int_{\Sigma_\e}\big( |\nabla u_n^0|^2
  + (c_\alpha-\mu_n^0)  |u_n^0|^2\big) \dx +\alpha\int_{\partial \Sigma_\e}  |u_n^0|^2\ds-\mc{T}_\e,
\end{align}
where $\nu$ is the outer normal derivative to $\Omega_\e$ on
$\partial \Sigma_\e$.  Plugging \eqref{proof_eigen_var_abs_3},
\eqref{proof_eigen_var_abs_4}, and \eqref{proof_eigen_var_abs_5} into
\eqref{proof_eigen_var_abs_2}, and taking into account that
  $c_\alpha-\mu_n^0=-\lambda_n^0$,
we obtain \eqref{eq_eigen_var_abs}.

For every $\e\in(0,1)$, let   $u_n^\e$ be  the eigenfunction of
\eqref{prob_robin_perturbed}, associated with the eigenvalue $\la_n^\e$,
chosen as in \eqref{hp_u_ne}. We consider the orthogonal projection
onto the eigenspace associated with the eigenvalue $\la_n^\e$, i.e.
\begin{equation*}
  \Pi_\e(v):= (v, u_n^\e)_{L^2(\Omega_\e)}  u_n^\e \quad \text{for
    every }   v \in L^2(\Omega_\e).
\end{equation*}
By Theorem \ref{theorem_expansion_very_abs}, \eqref{ineq_norm_qe},
\eqref{def_Te}, and Proposition
\ref{prop_norm_Ve_L2_H1}
\begin{align}
  &\|u_n^0-V_\e-\Pi_\e(u_n^0-V_\e)\|_{H^1(\Omega_\e)}^2=o\left(\mc{T}_\e\right)
    \quad \text{as } \e \to 0^+,\label{eq:eig11}\\
  \label{eq:instead}&
  \|u_n^0-\Pi_\e(u_n^0-V_\e)\|_{H^1(\Omega_\e)}^2
                      =O\left(\mc{T}_\e\right)\quad \text{as } \e \to 0^+,\\
  &\|u_n^0-\Pi_\e(u_n^0-V_\e)\|_{L^2(\Omega_\e)}^2=o\left(\mc{T}_\e\right)
    \quad \text{as } \e \to 0^+.\label{proof_eigen_var_abs_8}
\end{align}
In particular, by \eqref{eq:instead} and Proposition
\ref{prop_estimates_torsion},
\begin{equation}\label{proof_eigen_var_abs_9}
  \lim_{\e \to 0^+} \|u_n^0-\Pi_\e(u_n^0-V_\e)\|_{H^1(\Omega_\e)}=0.
\end{equation}
It follows that, thanks to \eqref{hp_u_ne} and
the simplicity of $\la_n^\e$,
\begin{equation}\label{eq:une}
u_n^\e=\frac{\Pi_\e(u_n^0-V_\e)}{\|\Pi_\e(u_n^0-V_\e)\|_{L^2(\Omega_\e)}}
\end{equation}
for all $\e>0$ sufficiently small. Hence
\begin{equation}\label{proof_eigen_var_abs_11}
  \|u_n^\e - \Pi_\e(u_n^0 -V_\e)\|^2_{H^1(\Omega_\e)}=
  \frac{\big|1-\|\Pi_\e(u_n^0-V_\e)\|_{L^2(\Omega_\e)}\big|^2}
  {\|\Pi_\e(u_n^0-V_\e)\|^2_{L^2(\Omega_\e)}}
 \|\Pi_\e(u_n^0-V_\e)\|^2_{H^1(\Omega_\e)}.
\end{equation}
Furthermore, in view of \eqref{limit_Ve_norm_0},
\eqref{proof_eigen_var_abs_8}, and the Cauchy-Schwarz inequality,
\begin{align}
  \notag\|\Pi_\e(u_n^0&-V_\e)\|^2_{L^2(\Omega_\e)}\\
  \notag &=1-\int_{\Sigma_\e} |u_n^0|^2 \dx +
           \|u_n^0-\Pi_\e(u_n^0-V_\e)\|^2_{L^2(\Omega_\e)}
           -2(u_n^0-\Pi_\e(u_n^0-V_\e),u_n^0)_{L^2(\Omega_\e)}\\
\label{proof_eigen_var_abs_12}&=1-\int_{\Sigma_\e} |u_n^0|^2 \dx +o(\mc{T}_\e^{1/2}).
\end{align}
Hence, plugging \eqref{proof_eigen_var_abs_12} in
\eqref{proof_eigen_var_abs_11} we obtain
  \begin{align*}
 \|u_n^\e - \Pi_\e(u_n^0
   -V_\e)\|^2_{H^1(\Omega_\e)}&=o(\mc{T}_\e)
   +O(\|u_n^0\|_{L^2(\Sigma_\e)}^4)  +o\left(\sqrt{\mc{T}_\e}\|u_n^0\|_{L^2(\Sigma_\e)}^2\right)\\
    &=o(\mc{T}_\e)
   +O(\|u_n^0\|_{L^2(\Sigma_\e)}^4)
  \end{align*}
as $\e \to 0^+$,
since
$\|\Pi_\e(u_n^0-V_\e)\|^2_{H^1(\Omega_\e)}=\|u_n^0\|_{H^1(\Omega)}^2+o(1)$ in view of
\eqref{proof_eigen_var_abs_9}.  By the Cauchy-Schwarz
inequality and \eqref{eq:eig11}, it then follows that 
\begin{align}
  \notag\|u_n^\e - (u_n^0 -V_\e)\|^2_{H^1(\Omega_\e)}&=\|u_n^0 -V_\e -
                                                       \Pi_\e(u_n^0 -V_\e)\|^2_{H^1(\Omega_\e)}
+\|u_n^\e -\Pi_\e(u_n^0 -V_\e)\|^2_{H^1(\Omega_\e)}\\
  \notag&\qquad -2\big(u_n^\e -\Pi_\e(u_n^0 -V_\e),u_n^0 -V_\e -
          \Pi_\e(u_n^0 -V_\e)\big)_{H^1(\Omega_\e)}\\
\label{eq:stueu0Ve}&=o(\mc{T}_\e)+O\big(\|u_n^0\|_{L^2(\Sigma_\e)}^4\big) \quad \text{as } \e \to 0^+,
\end{align}
so that \eqref{eq_eigenfun_abs_with_Ve} is proved.

To prove \eqref{eq_eigenfun_abs}, we observe that, by
\eqref{ineq_norm_qe}, \eqref{def_Te}, Proposition \ref{prop_traces},
and Proposition \ref{prop_norm_Ve_L2_H1},
\begin{equation}\label{eq:expVe}
\|V_\e\|^2_{H^1(\Omega_\e)} =
\mc{T}_\e+(1-c_\alpha)\|V_\e\|_{L^2(\Omega_\e)}^2
-\alpha\int_{\partial \Omega} V_\e^2 \ds
-\alpha\int_{\partial \Sigma_\e} V_\e^2 \ds=\mc{T}_\e+o\left(\mc{T}_\e\right)
\end{equation}
as  $\e \to 0^+$.
Moreover, from  \eqref{eq:eig11}, \eqref{eq:une},  and
 \eqref{proof_eigen_var_abs_12}, it follows that
\begin{align}
\notag   &(u_n^\e-(u_n^0-V_\e),V_\e)_{H^1(\Omega_\e)}\\
 \notag  &=
     \frac{\|\Pi_\e (u_n^0-V_\e)\|_{L^2(\Omega_\e)}-1}
     {\|\Pi_\e (u_n^0-V_\e)\|_{L^2(\Omega_\e)}}\,
     \|V_\e\|_{H^1(\Omega_\e)}^2\\
 \notag &\qquad
+     
   \frac{
   \Big(\Pi_\e (u_n^0-V_\e)-(u_n^0-V_\e)+\left(1-\|\Pi_\e
     (u_n^0-V_\e)\|_{L^2(\Omega_\e)}\right)u_n^0,
     V_\e\Big)_{H^1(\Omega_\e)}}{\|\Pi_\e (u_n^0-V_\e)\|_{L^2(\Omega_\e)}}\\
   \notag&=
\frac{\|\Pi_\e (u_n^0-V_\e)\|_{L^2(\Omega_\e)}-1}
     {\|\Pi_\e (u_n^0-V_\e)\|_{L^2(\Omega_\e)}}\,
     \big( \mc{T}_\e+o\left(\mc{T}_\e\right)\big)+
\frac{
   \Big(\Pi_\e (u_n^0-V_\e)-(u_n^0-V_\e),
     V_\e\Big)_{H^1(\Omega_\e)}}{\|\Pi_\e (u_n^0-V_\e)\|_{L^2(\Omega_\e)}}\\
\notag&\qquad +\frac{1-\|\Pi_\e
     (u_n^0-V_\e)\|_{L^2(\Omega_\e)}}{\|\Pi_\e (u_n^0-V_\e)\|_{L^2(\Omega_\e)}}
(u_n^0,   V_\e)_{H^1(\Omega_\e)}\\
 \label{eq:stl}&   =o\left(\mc{T}_\e\right)+O\left(\mc{T}_\e^{1/2}\|u_n^0\|_{L^2(\Sigma_\e)}^2\right)\quad\text{as   }
   \e\to0^+.
\end{align}
By \eqref{eq:stueu0Ve}, \eqref{eq:expVe}, and \eqref{eq:stl} we finally conclude that 
\begin{align*}
\|u_n^\e - u_n^0\|^2_{H^1(\Omega_\e)}
&=\|u_n^\e -
  (u_n^0-V_\e)\|^2_{H^1(\Omega_\e)}
+\|V_\e\|^2_{H^1(\Omega_\e)}
-2(u_n^\e - (u_n^0-V_\e),V_\e)_{H^1(\Omega_\e)}\\
&=\mc{T}_\e+o\left(\mc{T}_\e\right)+O(\norm{u_n^0}_{L^2(\Sigma_\e)}^4)
  +O\left(\mc{T}_\e^{1/2}\|u_n^0\|_{L^2(\Sigma_\e)}^2\right) \quad \text{as } \e \to 0^+,
\end{align*}
    thus completing the proof of  \eqref{eq_eigenfun_abs}.
\end{proof}

An asymptotic expansion for some of the terms appearing in
\eqref{eq_eigen_var_abs} is provided in the following proposition,
enabling us to derive Theorems
\ref{theor_asympotic_eigen_precise_u_not_0} and
\ref{theor_asympotic_eigen_precise_u_0_dim2} directly from 
expansion \eqref{eq_eigen_var_abs}.

\begin{proposition}\label{prop_asympotic_sigma_and_boundary_sigma}
Let $\ell\geq 1$ be the vanishing order of $u_n^0-u_n^0(0)$ at $0$. 
Then, as $\e\to0^+$,
\begin{align*}
  \int_{\Sigma_\e}|\nabla u_n^0|^2\dx&= \e^{N+2\ell-2}\int_{\Sigma}|\nabla P_\ell^{u_n^0,0}|^2\dx+ O(\e^{N+2\ell-1}),\\
  \int_{\Sigma_\e}|u_n^0|^2\dx&=\e^N|u_n^0(0)|^2\,|\Sigma|_N+\e^{N+2\ell}\int_{\Sigma}|P_\ell^{u_n^0,0}|^2\dx
                                +2\e^{N+\ell}\,u_n^0(0)\int_{\Sigma}P_\ell^{u_n^0,0}\dx\\
&\quad +u_n^0(0) O(\e^{\ell+1+N})+O(\e^{2\ell+1+N}),\\
  \int_{\partial \Sigma_\e}  |u_n^0|^2\ds&=\e^{N-1}|u_n^0(0)|^2
                                           \mc{H}^{N-1}(\partial
                                           \Sigma)
                                           +\e^{N+2\ell-1}\int_{\partial
                                       \Sigma}|P_\ell^{u_n^0,0}|^2\ds\\
                                     &\quad
                                       +2\e^{N+\ell-1}\,u_n^0(0)\int_{\partial
                                       \Sigma}P_\ell^{u_n^0,0}\ds+u_n^0(0) O(\e^{\ell+N})+O(\e^{2\ell+N}).
\end{align*} 
\end{proposition}
\begin{proof}
By standard elliptic regularity theory, $u_n^0$ is a smooth function; hence
\begin{align*}
u_n^0(x)-u_n^0(0)=P_\ell^{u_n^0,0}(x)+O(|x|^{\ell+1}) \quad \text{and}\quad 
\nabla u_n^0(x)=\nabla P_\ell^{u_n^0,0}(x)+O(|x|^\ell)\quad\text{as }|x|\to0.
\end{align*}
Then, the desired estimates follow directly from
a change of variables.
\end{proof}

\begin{proof}[Proof of Theorems
  \ref{theor_asympotic_eigen_precise_u_not_0} and
  \ref{theor_asympotic_eigen_precise_u_0_dim2}] Theorems
  \ref{theor_asympotic_eigen_precise_u_not_0} and
  \ref{theor_asympotic_eigen_precise_u_0_dim2} are now a simple
  consequence of Theorem \ref{theo_eigen_var_abs}, Proposition
  \ref{prop_estimates_torsion}, and Proposition
  \ref{prop_asympotic_sigma_and_boundary_sigma}.
\end{proof}

\section{Blow-up analysis for $N\ge 3 $}\label{sec_blow_up}

In this section,  we use a blow-up analysis to  derive the asymptotics
of $\mc{T}_\e$, as $\e \to 0^+$, for $N \ge 3$.

The following proposition describes the asymptotic behavior of $u_n^0$
near $x_0=0\in\Omega$; it directly follows from analyticity of
$u_n^0$.
\begin{proposition}\label{prop_vanish_u}
  If $u_n^0$ vanishes of order $k\geq 0$ at $x_0=0\in\Omega$, then,
  for every $R>0$,
  \begin{align*}
    r^{-k}u_n^0(rx)\rightarrow
    P_k^{u_n^0,0}(x) \quad \text{ uniformly in } \overline{B_R} \text{
    and in } H^1(B_R)
 \end{align*}  
 as $r\rightarrow 0^+$; furthermore, $P_k^{u_n^0,0}$ is a harmonic and
 homogeneous polynomial of degree $k$.

 If $u_n^0-u_n^0(0)$ vanishes
 of order $\ell\geq 1$ at $x_0=0\in\Omega$, then, for every $R>0$,
\begin{align*}
  \frac{u_n^0(rx)-u_n^0(0)}{r^\ell}\rightarrow P_\ell^{u_n^0,0}(x) \quad
  \text{ uniformly in } \overline{B_R}  \text{ and in } H^1(B_R) 
\end{align*} 
and
\begin{align*}
 r^{-\ell+1} \nabla u_n^0(rx)\rightarrow \nabla
  P_\ell^{u_n^0,0}(x)
  \quad \text{ uniformly in } \overline{B_R}  \text{ and in } H^1(B_R;\R^N) 
\end{align*} 
as $r\rightarrow 0^+$.
\end{proposition}

The following Hardy-type inequality serves as an essential tool for
conducting blow-up analysis, especially in the identification of limit
profiles. We refer to \cite[Lemma~5.3]{FLO_Neumann} for a proof.
\begin{lemma}\label{Hardy}
  Let $N\geq 3$ and $E$ be an open Lipschitz subset of $\RN$ such that
  $\overline{E}\subset B_{R_0}$ and
$B_{R_0}\setminus \overline{E}$ is connected,
  for some $R_0>0$. Then there exists
  $C_H= C_H(N, E)>0$ such that
\begin{align*}
  \int_{B_R\setminus E}\frac{u^2}{|x|^2} \dx \leq C_H
  \left(\int_{B_R\setminus E} |\nabla u|^2\dx +\frac{1}{R^2}
  \int_{B_R\setminus E} u^2\dx\right)
\end{align*}
for all $u\in H^1(B_R\setminus \overline{E})$ and $R>2R_0$. Moreover,
for all $u\in C_c^\infty (\RN \setminus E)$,
there holds
\begin{align*}
\int_{\RN\setminus E}\frac{u^2}{|x|^2} \dx \leq C_H\int_{\RN\setminus E} |\nabla u|^2\dx.
\end{align*}
\end{lemma}
For every open
Lipschitz set $E\subset \RN$
 such that
  $\overline{E}\subset B_{R_0}$ and
$B_{R_0}\setminus \overline{E}$ is connected
  for some $R_0>0$,  let us consider the functional space
$\mathcal{D}^{1,2}(\RN\setminus E)$, defined as the completion of
$C_c^\infty(\RN \setminus E)$ with respect to the norm
\begin{align*}
  \|u\|_{\mathcal{D}^{1,2}(\RN\setminus E)}:=
  \bigg(\int_{\RN\setminus E}|\nabla u|^2\dx\bigg)^{\frac{1}{2}}.
\end{align*}
The Hardy inequality provided by the above lemma implies that
\begin{equation*}
  u\mapsto  \left(\int_{\RN\setminus E}\bigg(|\nabla u|^2+\frac{u^2}{|x|^2}\bigg)\dx\right)^{\frac{1}{2}},
  \end{equation*}
  is an equivalent norm on 
$\mathcal{D}^{1,2}(\RN\setminus E)$, which can be characterized also
as 
\begin{align*}
\mathcal{D}^{1,2}(\RN\setminus E)=\left\{u\in W^{1,1}_{\rm loc}(\RN
  \setminus E):\,
  \int_{\RN\setminus E} \bigg(|\nabla u|^2+\frac{u^2}{|x|^2}\bigg)\dx<\infty\right\}.	
\end{align*}
Let $h$ and $P_h^{u_n^0,0}$ 
be as in \eqref{def_h}  and
\eqref{def_P_ui}, respectively, with $x_0=0\in\Omega$.
We define the functional $\widetilde{J}: \mathcal{D}^{1,2}(\RN
\setminus \Sigma) \to \R$ as  
\begin{equation*}
  \widetilde{J}(v):= \frac{1}{2}\int_{\R^N\setminus \Sigma}|\nabla
  v|^2\dx
  -\int_{\partial \Sigma}\big(\partial_\nu P_h^{u_n^0,0} +\alpha\, u_n^0(0)\big) v \ds.
\end{equation*}
\begin{proposition}\label{prop_min_tilde_J}
There exists a unique $\widetilde{V} \in
\mathcal{D}^{1,2}(\R^N\setminus \Sigma)$ such that
\begin{equation*}
  \widetilde{J}(\widetilde{V})=
  \min\{\widetilde{J}(v): v \in \mathcal{D}^{1,2}(\R^N\setminus \Sigma)\}.
\end{equation*}
Furthermore, $\widetilde{V}\not\equiv0$ and $\widetilde{V}$ is the
unique $\mathcal{D}^{1,2}(\R^N\setminus \Sigma)$-function satisfying 
\begin{equation}\label{swrtel-rp}
  \int_{\R^N\setminus \Sigma}\nabla \widetilde{V}\cdot \nabla \varphi
  \dx    =
  \int_{\partial \Sigma}\left(\partial_\nu P_h^{u_n^0,0} +\alpha\,
    u_n^0(0)\right)\varphi \ds
\end{equation}
for all $\varphi \in \mathcal{D}^{1,2}(\R^N\setminus  \Sigma)$.
\end{proposition}
\begin{proof}
  The proof of the existence of a unique minimizer $\widetilde{V}$ can
  be obtained by standard variational methods. To verify that
  $\widetilde{V}\not\equiv0$, we argue by contradiction. If
  $\widetilde{V}\equiv0$, \eqref{swrtel-rp} would imply that
  $\partial_\nu P_h^{u_n^0,0} +\alpha u_n^0(0)=0$ on
  $\partial\Sigma$. If $u_n^0(0)=0$, then $\Delta P_h^{u_n^0,0}=0$ and
  \begin{equation*}
    0=\int_{\partial\Sigma}(\partial_\nu
  P_h^{u_n^0,0}) P_h^{u_n^0,0}\ds=-\int_\Sigma|\nabla
  P_h^{u_n^0,0}|^2\dx;
\end{equation*}
this would imply that $P_h^{u_n^0,0}$ is constant, contradicting the
fact that $P_h^{u_n^0,0}$ is a polynomial of degree $h\geq1$.
On the other hand, if  $u_n^0(0)\neq0$, then $h=1$ and
$P_h^{u_n^0,0}(x)=\nabla u_n^0(0)\cdot x$; hence, by the Divergence Theorem,
  \begin{equation*}
    \alpha u_n^0(0) \mc{H}^{N-1}(\partial \Sigma)=-
    \int_{\partial\Sigma}\nabla u_n^0(0)\cdot \nu\ds=0,
  \end{equation*}
  giving rise to a contradiction in view of \eqref{eq:alpha-non-zero}.
\end{proof}
  We observe that \eqref{swrtel-rp} is the weak formulation of
\eqref{srtel-rp}.  We are now ready to state the main result of this
section.
\begin{theorem}\label{theorem_blow_up}
  Let $h$ and $P_h^{u_n^0,0}$ be as in \eqref{def_h} and
  \eqref{def_P_ui}, respectively, with $x_0=0\in\Omega$.  If $N\geq 3$
  and $\widetilde{V}$ is the function introduced in Proposition
  \ref{prop_min_tilde_J}, then
\begin{align*}
  \lim_{\e \to 0^+}\e^{-N-2h+2}\te =
  \tr=\int_{\partial \Sigma}\big(\partial_\nu P_h^{u_n^0,0}
  +\alpha\, u_n^0(0)\big)\widetilde{V} \ds,
\end{align*}
where $\te$ is defined in \eqref{def_Te}.
 Furthermore, if $V_\e\in H^1(\Omega_\e)$ is as in Proposition
 \ref{propo_tau_domain} and 
\begin{equation}\label{blow_ip_scale}
	\widetilde{V}_\e (x):=\e^{-h}V_\e(\e x), \quad x\in
        \left(\tfrac{1}{\e}
          \Omega \right)\setminus \Sigma,
      \end{equation}
then
\begin{equation*}
  \widetilde{V}_\e \rightarrow \widetilde{V} \quad \text{ in }
  H^1 (B_R\setminus \overline{\Sigma}) \quad \text{as } \e\rightarrow 0^+,
\end{equation*}
for all $R>0$ such that $\overline{\Sigma}\subset B_R$.
\end{theorem}
\begin{proof}
  Let $r_0>0$ be as in \eqref{def_Sigma} with $x_0=0$. Let $R>r_0$ and
  $\e<\frac{r_0}{2R}$. As ${\frac{r_0}{\e}}>2r_0$ and
  $R<\frac{r_0}{2\e}<\frac{r_0}{\e}$, applying Lemma~\ref{Hardy} to
  the function $\widetilde{V}_\e$, by a change of variable we have
\begin{align*}
  \int_{B_R\setminus \Sigma}\bigg(& |\nabla \widetilde{V}_\e |^2
    +
    \frac{\widetilde{V}_\e^2}{|x|^2}\bigg)\dx \leq
    \int_{B_{\frac{r_0}{\e}}\setminus \Sigma}\bigg( |\nabla
    \widetilde{V}_\e |^2
    + \frac{\widetilde{V}_\e^2}{|x|^2}\bigg)\dx\\
  &\leq \int_{B_{\frac{r_0}{\e}}\setminus \Sigma} |\nabla
    \widetilde{V}_\e |^2\dx
    +C_H\bigg(\int_{B_{\frac{r_0}{\e}}\setminus \Sigma}|\nabla
    \widetilde{V}_\e|^2\dx +
    \frac{\e^2}{r_0^2}\int_{B_{\frac{r_0}{\e}}\setminus
    \Sigma}\widetilde{V}_\e^2\dx\bigg)\\
  & = \e^{-N-2h+2}\bigg( \int_{B_{r_0}\setminus \e \Sigma }|\nabla
    V_\e|^2\dx
    +C_H\int_{B_{r_0}\setminus \e \Sigma } \left(|\nabla
    V_\e|^2+r_0^{-2}V_\e^2\right)\dx
    \bigg)\\
  &  \leq C_1\e^{-N-2h+2}\bigg( \int_{\Omega_\e }|\nabla V_\e|^2\dx
    +\int_{\Omega_\e }V_\e^2\dx\bigg)\\
  & = C_1\e^{-N-2h+2} ||V_\e||^2_{H^1(\Omega_\e)},
\end{align*}
for some $C_1>0$ independent of $R$ and $\e$. Therefore, by
\eqref{ineq_norm_qe}, \eqref{def_Te}, and
Proposition~\ref{prop_estimates_torsion}, we deduce that, for every $0<\e<\frac{r_0}{2R}$,
\begin{align}\label{hardy-help}
  \int_{B_R\setminus \Sigma}\bigg( |\nabla \widetilde{V}_\e |^2
  + \frac{\widetilde{V}_\e^2}{|x|^2}\bigg)\dx \leq C,
\end{align}
where the above constant $C$ is independent of $R$ and $\e$.

This implies that, for every $R>r_0$, there exists $C_R>0$ such
  that 
\begin{equation}\label{eq:bound}
  ||\widetilde{V}_\e||_{H^1(B_R\setminus\overline{\Sigma})}<C_R, \quad
  \text{for all $\e\in \big(0,\tfrac{r_0}{2R}\big)$}.
\end{equation}
Hence, by a  diagonal argument,
for every sequence $\e_j\to0^+$, there exist a subsequence (still
denoted as $\{\e_j\}_{j \in \mb{N}}$) and
$W\in W^{1,1}_{\rm loc}(\RN\setminus\overline{\Sigma})$ such that, for every $R>r_0$,
$W\in H^1(B_R\setminus\overline{\Sigma})$ and
\begin{align}\label{wh1}
  \widetilde{V}_{\e_j}\rightharpoonup W \quad \text{ weakly  in }
  H^1(B_R\setminus\overline{\Sigma}) \quad \text{ as } j\rightarrow +\infty.
\end{align}
Therefore, by the Rellich-Kondrachov Theorem and the compactness of
the trace operator into $L^2(\partial \Sigma)$, we have, for all
$R>r_0$,
\begin{align}
  &\label{sl2}\widetilde{V}_{\e_j} \rightarrow W
    \quad \text{ in } L^2(B_R\setminus\overline{\Sigma}) \quad \text{ as } j\rightarrow +\infty,
	\\& \label{stl2}\widetilde{V}_{\e_j} \rightarrow W \quad
  \text{ in }
  L^2(\partial \Sigma) \quad \text{ as } j\rightarrow +\infty.
\end{align}
Hence, by \eqref{hardy-help} and the weak lower semicontinuity of 
norms, we have
\begin{equation*}
  \int_{B_R\setminus\Sigma} \bigg(|\nabla
  W|^2+\frac{W^2}{|x|^2}\bigg)\dx\leq \liminf_{j\to+\infty}
  \int_{B_R\setminus\Sigma} \bigg(|\nabla \widetilde{V}_{\e_j} |^2
  +\frac{\widetilde{V}_{\e_j} ^2}{|x|^2}\bigg)\dx\leq C
\end{equation*}
for every $R>r_0$; 
hence, being $C$ independent of $R$, we conclude that
\begin{equation*}
  \int_{\R^N\setminus\Sigma} \bigg(|\nabla
  W|^2+\frac{W^2}{|x|^2}\bigg)\dx\le C,
\end{equation*}
so that $W\in \mathcal{D}^{1,2}(\RN\setminus \Sigma)$.

Let $v\in C_c^\infty(\RN\setminus\Sigma)$. There exist
$R_v>r_0$ and $j_v\in\N$ such that
\begin{equation*}
  \mathop{\rm supp} v\subset B_{R_v}\setminus\Sigma \subset
  \tfrac{1}{\e_j}\Omega\setminus \Sigma\quad\text{for all }j\geq j_v.
\end{equation*} 
Since $V_\e$ satisfies \eqref{eq_Ve}, by the change
of variable $x=\e_j y$ we obtain
\begin{align}\label{s-1}
&	\int_{B_{R_v}\setminus\Sigma}\big( \nabla
                           \widetilde{V}_{\e_j}
                           \cdot \nabla v + \e_j^{2}c_\alpha
                           \widetilde{V}_{\e_j} v \big)
                           \dx\nonumber \\
  & =  -\alpha \e_j\int_{\partial\Sigma} 
    \widetilde{V}_{\e_j} v\ds
    + \int_{\partial \Sigma}
    \bigg(\frac{\nabla u_n^0(\e_j x)}{\e_j^{h-1}} \cdot \nu +
    \alpha\, \e_j \frac{u_n^0(\e_j x)}{\e_j^h}\bigg)\,v \ds.
\end{align} 
In both cases, $u_n^0(0)=0$ and $u_n^0(0)\neq0$,
by Proposition \ref{prop_vanish_u} we have the uniform convergence 
\begin{equation}\label{jconv}
  \frac{\nabla u_n^0(\e x)}{\e^{h-1}}\cdot \nu +
  \e \alpha\, \frac{u_n^0(\e x)}{\e^h} \to 
  \nabla P_h^{u_n^0,0}(x)\cdot \nu
  +\alpha\, u_n^0(0) , \quad \text{in } \partial \Sigma,
\end{equation}
as $\e\to0^+$. Consequently
\begin{equation}\label{s-3}
  \int_{\partial \Sigma} \bigg( \frac{\nabla u_n^0(\e_j
  x)}{\e_j^{h-1}}\cdot
  \nu +    \alpha\, \e_j \frac{u_n^0(\e_j x)}{\e_j^h}\bigg)\,v \ds
  \to
  \int_{\partial \Sigma} \big(
\partial_\nu P_h^{u_n^0,0}+\alpha\, u_n^0(0)
\big)\,v \ds  ,
\end{equation}
as $j\rightarrow +\infty$. 
Hence, applying \eqref{wh1}, \eqref{sl2}, \eqref{stl2},  and
\eqref{s-3}, we can pass to the limit in \eqref{s-1}, thus obtaining 
\begin{align*}
	\int_{\RN\setminus\Sigma} \nabla W\cdot \nabla v   \dx =
  \int_{\partial \Sigma}\big(\partial_\nu P_h^{u_n^0,0}
  +\alpha\, u_n^0(0)\big)  v \ds
  \quad\text{for all $v\in \mathcal{D}^{1,2}(\RN\setminus\Sigma)$}.
\end{align*}
By uniqueness of the solution to problem \eqref{srtel-rp}, we conclude that 
\begin{align*}
	W=\widetilde{V}.
\end{align*}
By Urysohn's subsequence principle, we conclude that the convergences
\eqref{wh1}, \eqref{sl2}, and \eqref{stl2} hold as $\e\rightarrow 0$,
and not just along the sequence $\{\e_j\}_{j \in \mb{N}}$.

Moreover, by a change of variable, \eqref{eq_Ve} tested with $V_\e$,
\eqref{def_Te}, \eqref{stl2}, and \eqref{jconv},
\begin{align*}
  \e^{-N-2h+2}\te&= \int_{\partial \Sigma_{\e}}
                   \e^{-N-2h+2}\left(\partial_\nu u_n^0  +
                   \alpha\, u_n^0\right) V_{\e} \ds\\
                 &= \int_{\partial \Sigma}
                   \e^{-h+1}\left(\partial_\nu u_n^0(\e x) +
                   \alpha \,u_n^0(\e x)\right) \widetilde{V}_{\e} \ds\\
                 & \rightarrow \int_{\partial \Sigma} \left(
                   \partial_\nu P_h^{u_n^0,0}
                   +\alpha \, u_n^0(0)\right)\widetilde{V}\ds 
  \quad \text{as } \e \to 0^+.
\end{align*}
To conclude, it only remains to prove the strong $H^1$ convergence. By
 scaling, $\widetilde{V}_\e$ weakly solves
\begin{equation*}
	\begin{dcases}
		-\Delta \widetilde{V}_\e +c_\alpha
                \e^2\widetilde{V}_\e = 0, &
                \text{in } \tfrac{1}{\e}\Omega\setminus\overline{\Sigma}, \\
		\partial_\nu \widetilde{V}_\e  +    \alpha\e
                \widetilde{V}_\e =0,
                & \text{on } \partial\big( \tfrac\Omega\e\big),\\
		\partial_\nu \widetilde{V}_\e  +    \alpha \e
                \widetilde{V}_\e
                =\e^{-h+1}\partial_\nu u_n^0(\e x)  +
                \e^{-h+1} \alpha u_n^0(\e x), & \text{on } \partial\Sigma.\\
	\end{dcases}
\end{equation*}
For $R>r_0$ fixed, let $\e\in(0,\frac{r_0}{2R})$, so that 
$B_R\setminus \Sigma \subset \frac{1}{\e}\Omega\setminus\Sigma$. Then,
testing the difference between the above equation  and
\eqref{srtel-rp} with $\widetilde{V}_\e-\widetilde{V}$ we obtain
\begin{align}
	\notag &\int_{B_R\setminus \Sigma}|\nabla
                 (\widetilde{V}_\e-\widetilde{V})|^2\dx \\
               & = -\e^2c_\alpha\int_{B_R\setminus \Sigma} \widetilde{V}_\e(
                 \widetilde{V}_\e-\widetilde{V})\dx+\int_{\partial
                 B_R}
                 (\widetilde{V}_\e-\widetilde{V})(\partial_\nu\widetilde{V}_\e-
                 \partial_\nu\widetilde{V})
                 \ds\nonumber\\
               \label{sth1}&\quad +\int_{\partial \Sigma}
           (\widetilde{V}_\e-\widetilde{V})
           \bigg(- \alpha \e \widetilde{V}_\e+
           \frac{\nabla u_n^0(\e x)}{\e^{h-1}}\cdot\nu  +
           \e \alpha \frac{u_n^0(\e x)}{\e^h}-\partial_\nu P_h^{u_n^0,0}-\alpha u_n^0(0) \bigg)\ds.
\end{align}
We notice that, if $0<\e<\frac{r_0}{2R}$, \eqref{def_Sigma} ensures
that
$B_{2R}\setminus B_{r_0} \subset
\frac{1}{\e}\Omega\setminus\Sigma$. Hence, $\widetilde{V}_\e$ weakly
solves $-\Delta \widetilde{V}_\e =-c_\alpha \e^2\widetilde{V}_\e$ in
$B_{2R}\setminus B_{r_0}$. Since the family
$\{\e^2\widetilde{V}_\e\}_{\e\in (0,\frac{r_0}{2R})}$ is bounded in
$L^2(B_{2R}\setminus B_{r_0})$ by \eqref{eq:bound}, by classical
elliptic regularity theory we have
$\{\widetilde{V}_\e\}_{0<\e<\frac{r_0}{2R}}$ is bounded in
$H^2\big(B_{\frac{3}{2}R}\setminus B_{\frac{r_0+R}{2}}\big)$ uniformly
with respect to $\e$. Therefore, passing to traces,
$\{\partial_\nu \widetilde{V}_\e\}_{0<\e<\frac{r_0}{2R}}$ is bounded
in $L^2(\partial B_R)$. Hence, using this, \eqref{jconv}, and the weak
convergence $\widetilde{V}_{\e}\rightharpoonup \widetilde{V}$ in
$H^1(B_R\setminus\overline{\Sigma})$, which implies strong
$L^2(B_R\setminus\overline{\Sigma})$-convergence and strong
convergence of traces in $L^2(\partial B_R)$ and in
$L^2(\partial \Sigma)$, the right hand side of \eqref{sth1} vanishes
in the limit as $\e\to0^+$.  This proves that
$\nabla \widetilde{V}_\e\rightarrow \nabla \widetilde{V}$ in
$L^2(B_R\setminus \overline{\Sigma})$. This fact, along with
\eqref{sl2}, gives $\widetilde{V}_\e \rightarrow \widetilde{V}$ in
$H^1 (B_R\setminus \overline{\Sigma})$. This concludes the proof.
\end{proof}

\begin{proof}[Proof of
  Theorem~\ref{theor_asympotic_eigen_precise_u_0_dim_big}]
  By translation, it is not restrictive to assume $x_0=0$.  The
  result is a direct consequence of Theorem \ref{theo_eigen_var_abs},
  Proposition \ref{prop_asympotic_sigma_and_boundary_sigma}, and
  Theorem \ref{theorem_blow_up}.
\end{proof}

Now, we study the rate of convergence of perturbed eigenfunctions,
proving Theorems
\ref{theor_asympotic_eigenfunction_precise_u_0_dim_big} and
\ref{theor_asympotic_eigenfunction_precise_u_0_dim_2}.
\begin{proof}[Proof of Theorems
  \ref{theor_asympotic_eigenfunction_precise_u_0_dim_big} and
  \ref{theor_asympotic_eigenfunction_precise_u_0_dim_2}]
  By translation, it is not restrictive to assume $x_0=0$.  Let
  us consider the function
  $\varphi_\e:=\Phi_\e-\Phi_\e^0+\widetilde{V}_\e$, where
  $\widetilde{V}_\e$ is defined in \eqref{blow_ip_scale}. By
  \eqref{eq_eigenfun_abs_with_Ve} and a change of variable,
\begin{equation}\label{eq:au1}
  \int_{\frac{1}{\e}\Omega\setminus \Sigma} |\nabla \varphi_\e|^2\dx
  +\e^2 \int_{\frac{1}{\e}\Omega\setminus \Sigma}\varphi_\e^2\dx
  =o\left(\e^{-N-2h+2}\mc{T}_\e\right)+O(\e^{-N-2h+2}\norm{u_n^0}_{L^2(\Sigma_\e)}^4),
\end{equation}
as $ \e \to 0^+$. If $u_n^0(0)=0$, then $h=k$, where $k$ is the
vanishing order of $u_n^0$ at $0$. Therefore, $u_n^0(x)=O(|x|^k)$ as $|x|\to 0$,
hence $\|u_n^0\|_{L^2(\Sigma_\e)}^4=O(\e^{2N+4k})=o(\e^{N+2h-2})$
as $\e\to 0^+$.
 On the other hand, in the case
  $u_n^0(0)\neq 0$, we have $h=1$ and
  $u_n^0(x)=u_n^0(0)+O(|x|)$ as $|x|\to
  0^+$; hence
  $\|u_n^0\|_{L^2(\Sigma_\e)}^4=O(\e^{2N})=o(\e^N)=o(\e^{N+2h-2})$
  as $\e\to 0^+$.
In both cases, we conclude that 
\begin{align}\label{rate_ef_1}
    \|u_n^0\|^4_{L^2(\Sigma_\e)}=o(\e^{N+2h-2}) \quad \text{ as} \quad \e\to 0^+.
\end{align}
In particular, we have proved Theorem
\ref{theor_asympotic_eigenfunction_precise_u_0_dim_2} in view of
Theorem \ref{theo_eigen_var_abs} and Proposition
\ref{prop_estimates_torsion}.

Furthermore, by \eqref{eq:au1}, Theorem~\ref{theorem_blow_up}, and
\eqref{rate_ef_1},
\begin{align}\label{rate_ef_2}
\int_{\frac{1}{\e}\Omega\setminus \Sigma} |\nabla \varphi_\e|^2\dx
  +\e^2
  \int_{\frac{1}{\e}\Omega\setminus \Sigma}\varphi_\e^2\dx \to 0 \quad
  \text{ as } \e\to 0^+.
\end{align}
Let $r_0>0$ be as in \eqref{def_Sigma}. Choose $R>r_0$ and
$\e<\frac{r_0}{2R}$. If $N\geq3$, we can apply Lemma~\ref{Hardy} to obtain
\begin{align*}
  \int_{B_R\setminus \Sigma}&\left( |\nabla \varphi_\e |^2 +
                              \frac{\varphi_\e^2}{|x|^2}\right)\dx \leq
                              \int_{B_{\frac{r_0}{\e}}\setminus
                              \Sigma}
                              \left( |\nabla \varphi_\e |^2 +
                              \frac{\varphi_\e^2}{|x|^2}\right)\dx\\
                            & \leq
                              \int_{B_{\frac{r_0}{\e}}\setminus
                              \Sigma}
                              |\nabla \varphi_\e |^2\dx+C_H\bigg(
                              \int_{B_{\frac{r_0}{\e}}\setminus
                              \Sigma}
                              |\nabla \varphi_\e|^2\dx
                              +\frac{\e^2}{r_0^2}
                              \int_{B_{\frac{r_0}{\epsilon}}\setminus
                              \Sigma}\varphi_\e^2\dx\bigg)\\
                            & \leq (C_H+1)\bigg( \int_{\frac{1}{\e}\Omega\setminus \Sigma
                              }|\nabla \varphi_\e|^2\dx
                              +\frac{\e^2}{r_0^2}\int_{\frac{1}{\e}\Omega\setminus \Sigma}\varphi_\e^2\dx
                              \bigg).
\end{align*}
From this and \eqref{rate_ef_2} it follows that
\begin{align*}
  \int_{B_R\setminus \Sigma}\left( |\nabla \varphi_\e |^2 +
  \frac{\varphi_\e^2}{|x|^2}\right)\dx \to 0 \quad \text{ as } \e\to 0^+,
\end{align*}
for every $R>0$ such that $\overline{\Sigma}\subset B_R$. Hence
\begin{equation}\label{eq:phi-eps-to-zero}
\varphi_\e\to 0 \quad \text{strongly in $H^1(B_R\setminus \overline{\Sigma})$ as
  $\e\to 0^+$}.
\end{equation}
 If $u_n^0(0)=0$, then $h$ is the vanishing
  order of $u_n^0$ at $0$ and, by Proposition~\ref{prop_vanish_u},
  $\Phi^0_\e\to P_h^{u_n^0,0}$ as $\e\to0^+$ in $H^1(B_R)$ for any
  $R>0$.  
  
  If $u_n^0(0)\neq0$, then $h=1$ and
  $\Phi^0_\e(x)=\e^{-1}(u_n^0(\e x)- u_n^0(0))\to \nabla u_n^0(0)\cdot
  x=P_1^{u_n^0,0}(x)$ as $\e\to0^+$ in $H^1(B_R)$ for any $R>0$.

This, together with \eqref{eq:phi-eps-to-zero} and Theorem
\ref{theorem_blow_up}, establishes \eqref{eqn-1-rate-ef}.

Finally, using \eqref{rate_ef_1} and Theorem~\ref{theorem_blow_up} in
\eqref{eq_eigenfun_abs}, we obtain
\begin{equation*}
    \lim_{\e\to0^+}\e^{-(N+2h-2)}\|u_n^\e - u_n^0\|^2_{H^1(\Omega_\e)}=
  \tr,
\end{equation*}
thus proving  \eqref{eqn-2-rate-ef}.
\end{proof}

\section{The case of a round hole}\label{sec_round_hole}
In this section, we focus on the case where 
\begin{equation*}
\Sigma= B_{r_1},
\end{equation*}
for some $0<r_1<r_0$.  In this case, we can compute $\tr$ explicitly
if $N \ge 3$, while we determine the exact asymptotic behavior of
$\mc{T}_\e$ if $N=2$.  We do this only in the case $u_n^0(0)=0$, since
the case  $u_n^0(0)\neq0$ is already completely covered by Theorem
\ref{theor_asympotic_eigen_precise_u_not_0}.

\subsection{Dimension $N\ge 3$}\label{subsec_round_hole_dim_big}
If $N\geq3$ and $u_n^0(0)=0$, the quantity $\tr$ defined in
  \eqref{def_torsion_global} coincides with the one that
  appears in the blow-up analysis carried out in \cite{FLO_Neumann}
  for the Neumann
  problem; in the case of spherical holes, it can be explicitly
  computed as below.
\begin{proposition}\label{prop_round_hole_dim_big}
If $u_n^0(0) = 0$, $N \ge 3$, and $k$ is the vanishing order of
$u_n^0$ at $0$, then 
 \begin{equation*}
   \tr=\frac{k^2r_1^{N+2k-2}}{N+k-2}\int_{\partial B_1}
   |P_k^{u_n^0,0}|^2\ds,
\end{equation*}  
where $\tr$ is as in \eqref{def_torsion_global}.
\end{proposition}
\begin{proof}
The proof is contained in \cite[Lemma 6.1]{FLO_Neumann}, up to
    a rescaling.
\end{proof}

We are now in a position to prove Theorem
\ref{theor_asympotic_eigen_precise_round_u_0_dim_big}.

\begin{proof}[Proof of Theorem \ref{theor_asympotic_eigen_precise_round_u_0_dim_big}]
  Without loss of generality, we assume $x_0=0$. By the divergence
  theorem, since $P_k^{u_n^0,0}$ is harmonic by Proposition
  \ref{prop_vanish_u},
\begin{align*}
\int_{B_{r_1}} |\nabla P_k^{u_n^0,0}|^2 \dx&=
\int_{B_{r_1}} \dive\big( P_k^{u_n^0,0} \nabla P_k^{u_n^0,0}\big)
                                             \dx\\
  &=
\int_{\partial B_{r_1}} P_k^{u_n^0,0} \partial_{\nu} P_k^{u_n^0,0} \,
\ds
=k r_1^{N+2k-2}\int_{\partial B_1} |P_k^{u_n^0}|^2 \, \ds.
\end{align*}
Therefore Theorem
\ref{theor_asympotic_eigen_precise_round_u_0_dim_big} follows from
Theorem \ref{theor_asympotic_eigen_precise_u_0_dim_big} and
Proposition \ref{prop_round_hole_dim_big}.
\end{proof}

\subsection{Dimension $N=2$}\label{subsec_round_hole_dim_2}
In dimension $N=2$, we can instead determine an asymptotic expansion of
$\mc{T}_\e$ as $\e \to 0^+$.  To this aim, we proceed as in
\cite[Section 6.2]{FLO_Neumann}.

Let $N=2$, $\Sigma=B_{r_1}$ for some $0<r_1<r_0$, $x_0=0\in\Omega$,
$u_n^0(0)=0$ with $k\geq1$ being the vanishing order of $u_n^0$ at
$0$. Let $P_k^{u_n^0,0}$ be as in \eqref{def_P_ui}.

For every $R>r_1$, $R\neq-k/\alpha$, and $\e>0$ sufficiently small,
the function $W_{\e,R}: B_R \setminus \e\overline{B_{r_1}}\to\R$,
given in polar coordinates by
\begin{equation}\label{eq_Wer_explicit}
  W_{\e,R}(r,\theta)=
  r_1^{2k}\frac{(\alpha R -k)r^k-(\alpha R +k)R^{2k}r^{-k}}{(\alpha R +k)R^{2k}+(\alpha R -k)r_1^{2k}\e^{2k}}
  P_k^{u_n^0,0}(\theta)\e^{2k},
\end{equation}
is well-defined. Moreover, $W_{\e,R}\in C^\infty(\overline{ B_R
  \setminus \e B_{r_1}})$. Since the polynomial $P_k^{u_n^0,0}$ is
harmonic and $k$-homogeneous, we have $-\Delta_{{\mathbb
    S}^{1}}P_k^{u_n^0,0}=k^2 P_k^{u_n^0,0}$ on ${\mathbb
    S}^{1}$; hence $W_{\e,R}$ is a classical solution to the problem
\begin{equation*}
\begin{cases}
  -\Delta W_{\e,R} =0, & \text{ in } B_R \setminus \e\overline {B_{r_1}},\\
  \partial_{\nu} W_{\e,R}+\alpha W_{\e,R} = 0, & \text{ on } \partial B_R,\\
  \partial_{\nu} W_{\e,R}= \partial_{\nu} P_k^{u_n^0,0}, & \text{ on }
  \e\partial B_{r_1}.
\end{cases}
\end{equation*}
In particular, we also have  $W_{\e,R} \in H^1(B_R\setminus
\e\overline{B_{r_1}})$ and 
\begin{equation}\label{eq_W_er}
\int_{B_R \setminus \e B_{r_1}} \nabla  W_{\e,R} \cdot  \nabla v 
\dx+\alpha\int_{ \partial B_R}  W_{\e,R} v  \ds
=\int_{ \partial (\e B_{r_1})}  (\partial_{\nu}  P_k^{u_n^0,0}) v  \ds
\end{equation}
for every $v\in H^1(B_R\setminus \e\overline{B_{r_1}})$.
Let us  define
\begin{equation}\label{def_torsion_neumann}
  \mc{T}_\e^R:=
  \int_{ \partial (\e B_{r_1})}  (\partial_{\nu}  P_k^{u_n^0,0})
  W_{\e,R} \, \ds=
  \int_{B_R \setminus \e B_{r_1}} |\nabla  W_{\e,R}|^2
  \dx+\alpha\int_{ \partial B_R}  W_{\e,R}^2\ds.
\end{equation}
\begin{remark}\label{rem:Wer}
  We observe that, in view of \eqref{eq_Wer_explicit},
\begin{equation}\label{eq_asym_W_er}
  W_{\e,R}(\e r_1,\theta)= -r_1^k P_k^{u_n^0,0}(\theta)\e^k +o(\e^k)
  \quad  \text{as }\e \to 0^+,
\end{equation}
uniformly with respect to $\theta \in \mb{S}^{1}$.
\end{remark}

\begin{proposition}\label{prop_asymptotics_torsion_neumann}
  Let $u_n^0(0)=0$ and $k\geq1$ be the vanishing order of $u_n^0$ at
  $0$.
  Then, for every
  $R>r_1$, $R\neq-k/\alpha$, we have 
\begin{equation}\label{eq_asymptotics_torsion_Wer}
  \mc{T}_\e^R=k \pi r_1^{2k} \left(\left|
      \pd{^ku_n^0}{x^k_1}(0)\right|^2+
    \frac{1}{k^2}\left|\pd{^ku_n^0}{x^{k-1}_1\partial
        x_2}(0)\right|^2\right)
  \e^{2k}+o( \e^{2k}) \quad \text{as } \e \to 0^+.
\end{equation}
Furthermore,
\begin{align}
  &\norm{W_{\e,R}}_{L^2(B_R\setminus \e\overline{B_{r_1}})}=O( \e^{k+1}) \quad
    \text{as } \e \to 0^+,
    \label{eq_asymptotics_norm_L2_Wer}\\
  &\norm{W_{\e,R}}_{H^1(B_R\setminus \e\overline{B_{r_1}})}=O( \e^{k})
    \quad
    \text{as } \e \to 0^+.\label{eq_asymptotics_norm_H1_Wer}
\end{align}
\end{proposition}
\begin{proof}
Since the polynomial $P_k^{u_n^0,0}$ is  harmonic and $k$-homogeneous, we have necessarily
  \begin{equation*}
    P_k^{u_n^0,0}(r\cos t ,r\sin t)=r^k\bigg( \frac{\partial^k
      u_n^0}{\partial x_1^k}(0)\cos(kt)+
\frac1k \frac{\partial^k
      u_n^0}{\partial x_1^{k-1}\partial x_2}(0)\sin(kt)\bigg).
  \end{equation*}
  By   \eqref{eq_asym_W_er} and a change of variables 
\begin{align*}
  \int_{\partial B_{\e r_1}}\partial_\nu P_k^{u_n^0,0}  W_{\e, R}\ds
  &=-k \int_{\partial B_1}  P_k^{u_n^0,0}(\e  r_1 \theta)W_{\e, R}(\e  r_1\theta) \ds \\
  &=kr_1^{2k} \e^{2k}\int_{\partial B_1} |P_k^{u_n^0,0}(\theta)|^2\ds+o(\e^{2k})\\
  &= k\pi r_1^{2k}\e^{2k} \bigg(\bigg|
    \pd{^ku_n^0}{x^k_1}(0) \bigg|^2+\frac{1}{k^2}
    \bigg|\pd{^ku_n^0}{x^{k-1}_1\partial x_2}(0) \bigg|^2 \bigg)
    +o( \e^{2k})
\end{align*}
as $\e \to 0^+$, thus obtaining \eqref{eq_asymptotics_torsion_Wer}.
Furthermore, \eqref{eq_asymptotics_norm_L2_Wer} and
\eqref{eq_asymptotics_norm_H1_Wer} follow from
\eqref{eq_Wer_explicit}.
\end{proof}

Let us now consider the problem as in
Proposition~\ref{propo_tau_domain}, for the domain $\Omega=B_R$ and
$\Sigma=B_{r_1}$ with $R>r_1$, replacing the boundary Robin datum
$u_n^0$ with its approximation $P^{u_n^0,0}_k$. More precisely, let
$ Z_{\e,R}\in H^1(B_R\setminus \e \overline{B_{r_1}})$ be the unique
solution of the problem
\begin{equation*}
\begin{dcases}
  -\Delta  Z_{\e,R} +c_\alpha Z_{\e,R} = 0, & \text{in } B_R\setminus
  \e \overline{B_{r_1}}, \\
\partial_\nu Z_{\e,R}  +    \alpha Z_{\e,R} =0, & \text{on } \partial B_R,\\
\partial_\nu Z_{\e,R}  +    \alpha Z_{\e,R} =\partial_\nu P_k^{u_n^0,0}
+    \alpha P_k^{u_n^0,0}, & \text{on } \e \partial  B_{r_1},\\
\end{dcases}
\end{equation*}
that is, for every
$\varphi \in H^1(B_R\setminus  \e \overline{B_{r_1}})$,
\begin{multline}\label{test-1}
	\int_{B_R\setminus B_{\e r_1}} (\nabla Z_{\e,R}\cdot \nabla
        \varphi +
        c_\alpha Z_{\e,R} \varphi ) \dx +\alpha
        \int_{\partial (B_R\setminus B_{\e r_1})} Z_{\e,R}\,\varphi \ds \\
        = \int_{\partial B_{\e r_1}} \left(\partial_\nu P_k^{u_n^0,0}
          +    \alpha P_k^{u_n^0,0}\right)\varphi \ds.
\end{multline}
Let
\begin{align}
\notag	\tr_{\e,R}&= \int_{\partial B_{\e r_1}} \left(\partial_\nu
                    P_k^{u_n^0,0}  +
                    \alpha P_k^{u_n^0,0}\right) Z_{\e,R} \ds
  \\
  \label{eq_torsion_pn_ball}&=	\int_{B_R\setminus B_{\e r_1}}
                              (|\nabla Z_{\e,R}|^2 + c_\alpha
                              Z_{\e,R}^2 ) \dx
                              +\alpha\int_{\partial (B_R\setminus
                              B_{\e r_1})}  Z_{\e,R}^2 \ds.	
\end{align}
Arguing as in Proposition \ref{prop_estimates_torsion} and using
\eqref{ineq_norm_qe}, it is easy to verify that, for any $\delta\in (0,1)$,
\begin{equation}\label{eq_norm_ZeR}
\norm{ Z_{\e,R}}_{H^1(B_R\setminus \e \overline{B_{r_1}})}=
O(\e^{k-\frac{\delta}{2}})
\quad \text{as } \e \to 0^+,
\end{equation}
where $k\geq1$ is the vanishing order of $u_n^0$ at $0$.

\begin{proposition}\label{prop_torsion_ball_with_P}
  Let $N=2$, $u_n^0(0)=0$ and $k$ be the vanishing order of $u_n^0$ at
  $0$. For every $R>r_1$, $R\neq-k/\alpha$, and $\e>0$ sufficiently
  small, let $\tr_{\e,R}$ and $\mc{T}_\e^R$ be as in
  \eqref{eq_torsion_pn_ball} and \eqref{def_torsion_neumann},
  respectively. Then
  \begin{align}\label{eq_torsion_Ver_Wer}
    \tr_{\e,R}=\mc{T}_\e^R+o(\e^{2k}) \quad \text{as }\e\to 0^+.
\end{align}
\end{proposition}
\begin{proof}
  Let us choose $\varphi=W_{\e,R}$ in \eqref{test-1}. By
  \eqref{def_torsion_neumann} we have
\begin{multline*}
    \int_{B_R\setminus B _{\e r_1}} (\nabla Z_{\e,R}\cdot \nabla
    W_{\e,R}
    + c_\alpha Z_{\e,R} W_{\e,R} ) \dx\\
    +\alpha\int_{\partial (B_R\setminus B_{\e r_1})} Z_{\e,R}W_{\e,R}
    \ds
    -\alpha\int_{\partial B_{\e r_1}}      P_k^{u_n^0,0}W_{\e,R} \ds=\mc{T}_\e^R.
\end{multline*}
On the other hand, choosing $v= Z_{\e,R}$ in \eqref{eq_W_er}, in view
of \eqref{eq_torsion_pn_ball} we obtain
\begin{equation*}
\int_{B_R \setminus  B _{\e r_1}} \nabla  W_{\e,R} \cdot
\nabla
Z_{\e,R}  \dx+\alpha \int_{\partial B_R}    W_{\e,R} Z_{\e,R} \ds
+\alpha \int_{\partial B_{\e r_1}} P_k^{u_n^0,0} Z_{\e,R} \ds=\tr_{\e,R}.
\end{equation*}
Hence
\begin{align*}
  \tr_{\e,R}-\mc{T}_\e^R&=\alpha \int_{\partial B_{\e r_1}}
  P_k^{u_n^0,0} \left( Z_{\e,R}+W_{\e, R}\right) \ds\\
&\qquad   -\alpha\int_{\partial B_{\e r_1}} Z_{\e,R}W_{\e,R} \ds- c_\alpha
  \int_{B_R\setminus B_{\e r_1}} Z_{\e,R} W_{\e,R} \dx.  
 \end{align*}  
 By the Cauchy-Schwarz inequality, \eqref{ineq_traces_2},
 \eqref{eq_asymptotics_norm_H1_Wer} and \eqref{eq_norm_ZeR}, there
 exists a constant 
 $C>0$, which does not depend on $\e$, such that
\begin{multline*}
  \bigg|\int_{\partial B_{\e r_1}} P_k^{u_n^0,0} \left( Z_{\e,R}+W_{\e,
      R}\right) \ds\bigg| \le C \e^{k+\frac{1}{2}} \left(\int_{\partial
      B_{\e r_1}}
    \left( Z_{\e,R}+W_{\e, R}\right)^2 \ds \right)^\frac{1}{2}\\
  \le C \sqrt{C_\delta} \e^{k+1-\frac{\delta}{2}}
  \big(\|Z_{\e,R}\|_{H^1(B_R\setminus \e \overline{B_{r_1}})} +\|W_{\e,
    R}\|_{H^1(B_R\setminus \e
    \overline{B_{r_1}})}\big)=O(\e^{2k+1-\delta})
  \quad \text{as } \e \to 0^+,
\end{multline*}
for any $\delta \in (0,1)$.  Similarly, by the Cauchy-Schwarz
inequality, \eqref{ineq_traces_2}, \eqref{eq_asymptotics_norm_H1_Wer}
and \eqref{eq_norm_ZeR},
\begin{equation*}
  \int_{\partial B_{\e r_1}} Z_{\e,R}W_{\e,R} \ds
  =O(\e^{2k+1-\frac{3}{2}\delta}) \quad \text{as } \e \to 0^+,
\end{equation*}
for any $\delta \in (0,1)$.  Furthermore, by the Cauchy-Schwarz
inequality, \eqref{eq_asymptotics_norm_L2_Wer} and
\eqref{eq_norm_ZeR},
\begin{equation*}
  \int_{B_R\setminus B_{\e r_1}} Z_{\e,R} W_{\e,R} \dx =
  O(\e^{2k+1-\frac{\delta}{2}})
  \quad \text{as } \e \to 0^+,
\end{equation*}
for any $\delta \in (0,1)$. Choosing $\delta<\frac23$, we thereby prove
\eqref{eq_torsion_Ver_Wer}.
\end{proof}

Let us now consider the problem as in
Proposition~\ref{propo_tau_domain} for $\Omega=B_R$ and
$\Sigma=B_{r_1}$, for some $R>r_1$. Let
$V_{\e, R}\in H^1(B_R\setminus \e\overline{B_{r_1}})$ be the unique
solution of the problem
\begin{equation*}
\begin{dcases}
  -\Delta V_{\e, R} +c_\alpha V_{\e, R} = 0, & \text{in } B_R\setminus
  \e\overline{B_{r_1}}, \\
  \partial_\nu V_{\e, R}  +    \alpha V_{\e, R} =0, & \text{on } \partial B_R,\\
  \partial_\nu V_{\e, R} + \alpha V_{\e, R} =
  \partial_\nu {u_n^0}  +    \alpha {u_n^0}, & \text{on } \e\partial B_{r_1},\\
\end{dcases}
\end{equation*}
that is, for every $\varphi \in  H^1(B_R\setminus \e\overline{B_{r_1}})$,
\begin{align*}
  \int_{B_R\setminus B_{\e r_1}} (\nabla V_{\e, R}\cdot \nabla
  \varphi +
  c_\alpha V_{\e, R} \varphi ) \dx +
  \alpha\int_{\partial (B_R\setminus B_{\e r_1})} V_{\e,
  R}\,\varphi \ds
  = \int_{\partial B_{\e r_1}} \left(\partial_\nu {u_n^0}  +    \alpha
  {u_n^0}\right)
  \varphi \ds.
\end{align*}
Also, we denote the corresponding torsional rigidity as $\mc{T}_{\e ,
  R}$, so that 
\begin{align}
\notag	\mc{T}_{\e,R}&= \int_{\partial B_{\e r_1}} \left(\partial_\nu {u_n^0}  +    \alpha {u_n^0}\right) V_{\e, R} \ds\\
  \label{final_torsion} &=
                          \int_{B_R\setminus  B_{\e r_1}} (|\nabla V_{\e,
                          R}|^2 +
                          c_\alpha  V_{\e, R}^2 ) \dx +\alpha
                          \int_{\partial (B_R\setminus  B_{\e r_1})}  V_{\e, R}^2 \ds.	
\end{align}

\begin{proposition}\label{prop_torsion_ball_with_u}
Let $N = 2$, $u_n^0(0)=0$, and  $k\geq1$ be the vanishing order of
$u_n^0$ at $0$. For every $R>r_1$, we have 
\begin{align}\label{eq_torsion_Zer_Ver}
\mc{T}_{\e,R}= \tr_{\e,R}+O(\e^{2k+1-\frac\delta2}) \quad \text{as }\e\to 0^+.
\end{align}
\end{proposition}
\begin{proof}
  Let $P_k^{u_n^0,0}, P_{k+1}^{u_n^0,0}$ be as in
  \eqref{def_P_ui}. Let us denote with $Y_{\e, R}$ the unique weak
  solution of the problem
\begin{equation*}
\begin{dcases}
  -\Delta Y_{\e, R} +c_\alpha Y_{\e, R} = 0, & \text{in } B_R\setminus
  \e\overline{B_{r_1}}, \\
  \partial_\nu Y_{\e, R}  +    \alpha Y_{\e, R} =0, & \text{on } \partial B_R,\\
  \partial_\nu Y_{\e, R} + \alpha Y_{\e, R} =\partial_\nu
  P_{k+1}^{u_n^0,0} +
  \alpha P_{k+1}^{u_n^0,0}, & \text{on } \e\partial B_{r_1}.\\
\end{dcases}
\end{equation*}
Let $\Psi_\e=V_{\e, R}- Z_{\e,R}- Y_{\e, R}\in H^1(B_R\setminus
\e\overline{B_{r_1}})$.
Then, for every $\varphi\in H^1(B_R\setminus \e\overline{B_{r_1}})$,
\begin{multline*}
     \int_{B_R\setminus B_{\e r_1}} (\nabla\Psi_\e\cdot \nabla \varphi
     +
     c_\alpha \Psi_\e \varphi ) \dx +\alpha
     \int_{\partial (B_R\setminus B_{\e r_1})} \Psi_\e\,\varphi \ds \\
     = \int_{\partial B_{\e r_1}} \left(\partial_\nu
       \big({u_n^0}-P_{k}^{u_n^0,0}
       -P_{k+1}^{u_n^0,0}\big)  +    \alpha
       \big({u_n^0}-P_{k}^{u_n^0,0}-P_{k+1}^{u_n^0,0}\big)\right)\varphi \ds.
\end{multline*}
In particular, testing with $\varphi=\Psi_\e$,  we have
\begin{align*}
     C_\alpha^{-1}\|\Psi_\e\|_{H^1(B_R\setminus
       \e\overline{B_{r_1}})}^2&\leq
     \int_{B_R\setminus B_{\e r_1}} (|\nabla\Psi_\e|^2 + c_\alpha
     \Psi_\e^2) \dx +
                                 \alpha\int_{\partial (B_R\setminus B_{\e r_1})} \Psi_\e^2 \ds \\
  &=
\int_{\partial B_{\e r_1}} \left(\partial_\nu
       \big({u_n^0}-P_{k}^{u_n^0,0}
       -P_{k+1}^{u_n^0,0}\big)  +    \alpha
       \big({u_n^0}-P_{k}^{u_n^0,0}-P_{k+1}^{u_n^0,0}\big)\right) \Psi_\e \ds
\end{align*}
by \eqref{ineq_norm_qe}.  Thanks to the Cauchy-Schwarz inequality,
the Taylor expansion of $u_n^0$ at $0$, and
\eqref{ineq_traces_2}, there exists a constant $C>0$, independent of
$\e$, such that
\begin{multline*}
  \int_{\partial B_{\e r_1}}
  \left(\partial_\nu
    \big({u_n^0}-P_{k}^{u_n^0,0}-P_{k+1}^{u_n^0,0}\big) + \alpha
    \big({u_n^0}-P_{k}^{u_n^0,0}-P_{k+1}^{u_n^0,0}\big)\right)\Psi_\e
  \ds \\
  \leq C \e^{k+\frac{3}{2}}\left(\int_{\partial
      B_{\e r_1}}\Psi_\e^2\ds\right)^{\frac{1}{2}}\leq
  C\sqrt{C_\delta}\e^{k+2-\frac{\delta}{2}}\|\Psi_\e\|_{H^1(B_R\setminus \e\overline{B_{r_1}})}
\end{multline*}
for any $\delta \in (0,1)$. Hence,
\begin{align}\label{fnl-1}
  \|\Psi_\e\|_{H^1(B_R\setminus \e\overline{B_{r_1}}) }=
  O(\e^{k+2-\frac{\delta}{2}})\quad \text{as } \e\to 0^+.
\end{align}
Arguing as in Proposition \ref{prop_estimates_torsion}, by
\eqref{ineq_norm_qe} we obtain
\begin{align}\label{fnl-2}
\|Y_{\e, R}\|_{H^1(B_R\setminus
  \e\overline{B_{r_1}})}=O(\e^{k+1-\frac{\delta}{2}})\quad
  \text{as } \e\to 0^+,
\end{align}
for any $\delta\in (0,1)$.
By  \eqref{ineq_norm_qe} and  \eqref{eq_torsion_pn_ball}, we have 
\begin{equation*}
\|Z_{\e,R}\|_{H^1(B_R\setminus \e\overline{B_{r_1}})}
=O(\widetilde{\mc{T}}_{\e,R}^{1/2})\quad\text{as }\e\to0^+.
\end{equation*}
Hence, by \eqref{eq_asymptotics_torsion_Wer} and Proposition
\ref{prop_torsion_ball_with_P}, it follows that
\begin{equation}\label{fnl-3}
\|Z_{\e,R}\|_{H^1(B_R\setminus \e \overline{B_{r_1}})} = O(\e^k) \quad \text{as } \e \to 0^+.
\end{equation}
Therefore, by  \eqref{fnl-1}, \eqref{fnl-2} and \eqref{fnl-3}, we have
\begin{align}\label{fnl-4}
    \|V_{\e, R}\|_{H^1(B_R\setminus \e \overline{B_{r_1}})}=O(\e^k) \quad \text{as } \e \to 0^+.
\end{align}
Furthermore, by the  Cauchy-Schwarz inequality
\begin{align*}
    &\|V_{\e, R}-Z_{\e,R}\|_{H^1(B_R\setminus
      \e\overline{B_{r_1}})}^2-
      \|Y_{\e, R}\|_{H^1(B_R\setminus \e\overline{B_{r_1}})}^2\\
    &=\big(\Psi_\e,V_{\e, R}-Z_{\e,R}+Y_{\e,
      R}\big)_{H^1(B_R\setminus \e\overline{B_{r_1}})}
  \\&\leq \|\Psi_\e\|_{H^1(B_R\setminus  \e\overline{B_{r_1}})}
  \left(\|V_{\e, R}\|_{H^1(B_R\setminus  \e\overline{B_{r_1}})}+
  \| Z_{\e,R}\|_{H^1(B_R\setminus  \e\overline{B_{r_1}})}+
  \| Y_{\e, R},\|_{H^1(B_R\setminus  \e\overline{B_{r_1}})}\right),
\end{align*}
so that, in view of \eqref{fnl-1}, \eqref{fnl-2}, \eqref{fnl-3}, and
\eqref{fnl-4},
\begin{align}\label{fnl-5}
  \|V_{\e, R}-Z_{\e,R}\|_{H^1(B_R\setminus
  \e\overline{B_{r_1}})}^2-
  \|Y_{\e, R}\|_{H^1(B_R\setminus \e\overline{B_{r_1}})}^2
  =O(\e^{2k+2-\frac{\delta}{2}}) \quad \text{as } \e \to 0^+,
\end{align}
for any $\delta\in (0,1)$.
From \eqref{fnl-2} and \eqref{fnl-5} it follows that
\begin{align}\label{fnl-6}
  \|V_{\e, R}-Z_{\e,R}\|_{H^1(B_R\setminus \e\overline{B_{r_1}})}=
  O(\e^{k+1-\frac{\delta}{2}}) \quad \text{as } \e \to 0^+,
\end{align}
for any $\delta\in (0,1)$.

By the  Cauchy-Schwarz inequality for $q_\e$,
  \eqref{ineq_qe_norm} for $\Omega=B_R$, \eqref{fnl-6}, \eqref{fnl-3},
  and \eqref{fnl-4}
\begin{align*}
  q_\e(V_{\e, R}-Z_{\e,R},V_{\e, R}+Z_{\e, R})
&\leq 
\sqrt{q_\e(V_{\e, R}-Z_{\e,R},V_{\e, R}-Z_{\e, R})}
\sqrt{q_\e(V_{\e, R}+Z_{\e,R},V_{\e, R}+Z_{\e, R})}\\
&  =O(\e^{2k+1-\frac\delta2})
  \quad \text{as } \e \to 0^+,
\end{align*}
for any $\delta\in (0,1)$, where $q_\e$ is as in \eqref{def_qe} for
the domain $B_R\setminus \e \overline{B_{r_1}}$.  Hence, by
\eqref{eq_torsion_pn_ball} and \eqref{final_torsion},
\begin{equation*}
      |\tr_{\e,R}-\mc{T}_{\e,R}|=|q_\e(Z_{\e, R},Z_{\e,
        R})-q_\e(V_{\e, R},V_{\e, R})|=|q_\e(V_{\e, R}-Z_{\e,R},V_{\e, R}+Z_{\e, R})|
      =O(\e^{2k+1-\frac\delta2})
    \end{equation*}
  as $\e\to0^+$, for any $\delta\in (0,1)$, thus proving \eqref{eq_torsion_Zer_Ver}.
\end{proof}

In order to complete the proof of Theorem
\ref{theor_asympotic_eigen_precise_round_u_0_dim_2}, we need the
following general lemma concerning the comparison of quadratic forms
with boundary terms.
For any open and connected Lipschitz  set $\Lambda\subset\R^N$ and any
constant $C>0$, we define
\begin{equation*}
  q_{\Lambda, C}: H^1(\Lambda)\to\R,\quad
  q_{\Lambda, C}(u):=\int_{\Lambda} \left(|\nabla u|^2+C u^2
  \right)\dx
  +\alpha\int_{\partial \Lambda} u^2 \ds.
\end{equation*}
\begin{lemma}\label{lemma_domains}
  Let $\Lambda_1, \Lambda_2$ be open and connected Lipschitz sets in
  $\R^N$ such that $\overline{\Lambda_1}\subset \Lambda_2$. For every
  $\alpha>0$ there exists a constant
  $d=d(\alpha,\Lambda_1,\Lambda_2)>0$, which depends only on
  $\Lambda_1$, $\Lambda_2$, and $\alpha$, such that
\begin{equation}\label{ineq_domains}
q_{\Lambda_1, C} (u)\le q_{\Lambda_2, C}(u)
\end{equation}
 for every
  $u \in H^1(\Lambda_2)$ and $C\geq d$.
\end{lemma}
\begin{proof}
Inequality \eqref{ineq_domains} can be rewritten as
\begin{equation*}
  \int_{\Lambda_2\setminus \Lambda_1} \left(|\nabla u|^2+C u^2
  \right)\dx +
  \alpha\int_{\partial \Lambda_2} u^2 \ds- \alpha\int_{\partial
    \Lambda_1}
  u^2\ds\ge 0
\end{equation*}
for any $u\in H^1(\Lambda_2)$. Hence, to conclude it is enough to show
that there exists a constant $d>0$, depending only on  $\Lambda_1$,
$\Lambda_2$, and $\alpha$, such that 
\begin{equation*}
\int_{\Lambda_2 \setminus\Lambda_1} \left(|\nabla u|^2+d u^2 \right)\dx
\ge
|\alpha| \left(\int_{\partial \Lambda_2} u^2 \ds+\int_{\partial
    \Lambda_1} u^2
  \ds\right)
\end{equation*}
for every $u\in H^1(\Lambda_2)$. The above inequality holds in view of
\cite[Theorem 18.1]{L_book_sobolev}.
\end{proof}

\begin{proposition}\label{prop_round_hole_dim_2}
  Let $N = 2$, $u_n^0(0)=0$ and $k$ be the vanishing order of $u_n^0$
  at $0$.  Then
 \begin{equation}\label{limit_torsion_annuli}
 \mc{T}_\e=k \pi r_1^{2k} \left(\left|
     \pd{^ku_n^0}{x^k_1}(0)\right|^2+\frac{1}{k^2}
   \left|\pd{^ku_n^0}{x^{k-1}_1\partial x_2}(0)\right|^2\right) \e^{2k}+o( \e^{2k})
 \quad \text{as } \e \to 0^+.
\end{equation}   
\end{proposition}
\begin{proof}
  For any $R>r_1$ with $R\neq -k/\alpha$, by
  Propositions~\ref{prop_torsion_ball_with_u} and
  \ref{prop_torsion_ball_with_P} we have
\begin{align*}
  \mc{T}_{\e,R}=\tr_{\e,R}+O(\e^{2k+1-\frac\delta2})=\mc{T}_\e^R+o(\e^{2k}) \quad \text{ as }\e\to 0^+.
\end{align*}
Hence, by Proposition~\ref{prop_asymptotics_torsion_neumann}, for any
$R>r_1$ with $R\neq -k/\alpha$ we have
\begin{equation}\label{eq:exp-tepsR}
\mc{T}_{\e,R}=k \pi r_1^{2k} \left(\left|
    \pd{^ku_n^0}{x^k_1}(0)\right|^2+\frac{1}{k^2}
  \left|\pd{^ku_n^0}{x^{k-1}_1\partial x_2}(0)\right|^2\right) \e^{2k}+o( \e^{2k}) \quad \text{as } \e \to 0^+.
\end{equation}
Let $R_2>R_1>r_0>r_1$ be such that
$\overline{B_{R_1}}\subset \Omega$ and $\overline{\Omega} \subset B_{R_2}$ and $R_1,R_2\neq -k/\alpha$.
Using the characterization of $\mc{T}_{\e,R_2}, \mc{T}_{\e,R_1},
  \mc{T}_{\e}$  given in \eqref{eq_chara_torsion} and choosing
  $c_\alpha$  larger than the constant $d$ given by Lemma
  \ref{lemma_domains}  applied first with $\Lambda_1=B_{R_1}$ and
  $\Lambda_2=\Omega$, and  then 
  with $\Lambda_1=\Omega$ and
  $\Lambda_2=B_{R_2}$, we obtain that
\begin{align*}
    \mc{T}_{\e,R_2}\leq \te \leq \mc{T}_{\e,R_1}.
\end{align*}
Estimate \eqref{limit_torsion_annuli} finally follows from
\eqref{eq:exp-tepsR} by comparison.
\end{proof}

\begin{proof}[Proof of Theorem \ref{theor_asympotic_eigen_precise_round_u_0_dim_2}]
 Without loss of generality, we assume $x_0=0$.
  We can proceed as in the proof of Theorem
  \ref{theor_asympotic_eigen_precise_round_u_0_dim_big} to obtain that
\begin{equation*}
\int_{B_{\e r_1}} |\nabla P_k^{u_n^0,0}|^2 \, \dx=k  r_1^{2k}\e^{2k}\int_{\partial B_1} |P_k^{u_n^0,0}|^2  \ds .
\end{equation*}
Then Theorem \ref{theor_asympotic_eigen_precise_round_u_0_dim_2}
follows from Theorem \ref{theo_eigen_var_abs}, Proposition
\ref{prop_asympotic_sigma_and_boundary_sigma}, and Proposition
\ref{prop_round_hole_dim_2}.
\end{proof}

\bigskip\noindent {\bf Acknowledgments.} 
The authors are supported by the MUR-PRIN project no. 20227HX33Z 
``Pattern formation in nonlinear phenomena'' granted by the European
Union - Next Generation EU.
V. Felli and G. Siclari are members of GNAMPA-INdAM.

\appendix
\section{} \label{sec_appendix}

In this appendix we present a simplified version of \cite[Theorem
3.1]{BS_quantitative}, adapted to our case, specifically the removal
of small sets in some bounded domain $\Omega\subset \R^N$.  An
  alternative proof of this can also be obtained quite directly  by
  applying \cite[Lemma A.1]{FLO_Neumann}, which revisits the
  well-known \emph{Lemma on small eigenvalues} by Y. Colin de Verdi\`ere
  \cite{ColindeV1986} in the context of simple eigenvalues.

\begin{theorem}\label{theorem_expansion_very_abs}
Let $\Omega\subset \R^N$ be a bounded open Lipschitz set. Let
  $\{K_\e\}_{\e \in [0,1]}$ be a family of compact sets such that
\begin{equation*}
  K_\e \subset \Omega \quad \text{for every } \e \in [0,1], \quad
  K_0=\emptyset,
  \quad \lim_{\e\to 0^+}|K_\e|_N=0,
\end{equation*}
and
  \begin{equation}\label{eq:comp-emb}
    \text{the embedding $H^1(\Omega_\e)\subset L^2(
      \Omega_\e)$ is compact for every $\e \in [0,1]$,}
  \end{equation}
where $\Omega_\e:= \Omega \setminus K_\e$.
For every $\e\in[0,1]$, let  $q_\e : H^1(\Omega_\e) \times
H^1(\Omega_\e) \to \R$
be a symmetric and continuous
bilinear form such that
\begin{equation*}
  q_\e(v,v) \ge \eta
  \norm{v}^2_{H^1(\Omega_\e)}
  \text{ for all } v \in  H^1(\Omega_\e),
\end{equation*}
for some constant $\eta>0$ that is independent of $\e\in[0,1]$.
Letting $\{\rho_j^\e\}_{j\in \mathbb{N}\setminus \{0\}}$ be the
eigenvalues of $q_\e $, repeated according to their multiplicity, we
assume that
\begin{equation*}
\lim_{\e \to 0}\rho_j^\e=\rho_j^0
\end{equation*}
for every $j \in \mathbb{N}\setminus \{0\}$.
Suppose that, for some $n\in \N\setminus\{0\}$,
\begin{equation*}
  \rho_n^0\text{ is simple},
\end{equation*}
and let $u_n^0$ be an associated eigenfunction such that $\int_{\Omega} |u_n^0|^2\dx=1$.
For every $\e\in (0,1]$, let 
\begin{equation*}
  L_\e: H^1(\Omega_\e)\to \R,\quad
  L_\e(v):= q_\e(u_n^0,v)-\rho_n^0 (u_n^0,v)_{L^2(\Omega_\e)}
\end{equation*}
and 
\begin{equation*}
 J_\e: H^1(\Omega_\e)\to \R,\quad J_\e(v):=\frac{1}{2}\, q_\e(v,v) -L_\e(v).
\end{equation*}
Assume that
\begin{equation*}
\lim_{\e \to 0^+} q_\e(V_\e, V_\e) =0,
\end{equation*}
where $V_\e$ is the unique solution to the minimization problem 
\begin{equation*}
  J_\e(V_\e)=\min_{v\in H^1(\Omega_\e)}J_\e(v).
\end{equation*}
Then 
\begin{equation*}
\rho_n^\e-\rho_n^0= 
\frac{\rho_n^0\int_{\Omega_\e} V_\e u_n^0 \dx +O(\norm{V_\e}^2_{L^2(\Omega_\e)})}
{\int_{\Omega_\e}|u_n^0|^2\dx+O(\norm{V_\e}_{L^2(\Omega_\e)})} \quad 
\text{as } \e \to 0^+.
\end{equation*}
Finally, as $\e\to0^+$,
\begin{align*}
  &q_\e\big(u_n^0-V_\e-\Pi_\e(u_n^0-V_\e),u_n^0-V_\e-\Pi_\e(u_n^0-V_\e)\big)
    =O\big(\|V_\e\|^2_{L^2(\Omega_\e)}\big), \\
  &q_\e(u_n^0-\Pi_\e(u_n^0-V_\e),u_n^0-\Pi_\e(u_n^0-V_\e))
    =q_\e(V_\e,V_\e)+O\big(\norm{V_\e}_{L^2(\Omega_\e)}(q_\e(V_\e,V_\e))^{\frac{1}{2}}\big),\\
  &\|u_n^0-\Pi_\e(u_n^0-V_\e)\|_{L^2(\Omega_\e)}^2=O\big(\|V_\e\|^2_{L^2(\Omega_\e)}\big),
\end{align*}
where, for every $\e\in(0,1]$, $\Pi_\e : L^2(\Omega_\e)\to
L^2(\Omega_\e)$ is  the projection map onto the eigenspace of $q_\e$
associated to the eigenvalue $\rho_n^\e$.
\end{theorem}
We observe that condition  \eqref{eq:comp-emb} is satisfied if
  $K_\e$ is the closure of an open  Lipschitz set; in particular, it holds for
  $K_\e=\overline\Sigma_\e$, with $\Sigma_\e$ satisfying
  \eqref{def_Sigma_e} for some Lipschitz bounded open set $\Sigma$.

\section{} \label{sec_appendix-extension}

In this appendix, we present a proof of Proposition \ref{prop_extension}.
The following lemma states that the Sobolev extension operator from a
connected domain into an interior hole undergoes an estimate not only of
the complete $H^1$-norms, but also of the $L^2$-norms of gradients.

\begin{lemma}\label{l:extension_1}
Let  
$\Sigma\subset\R^N$ be an open, bounded, and  Lipschitz set such that
\begin{equation*}
  \overline{\Sigma} \subset B_{r_0} \quad\text{and}\quad
  B_{r_0}\setminus \overline{\Sigma}\text{ is connected},  
\end{equation*}
for some $r_0>0$. There exist a linear bounded extension operator
$\mathsf{E}:H^1( B_{r_0}\setminus \overline{\Sigma})\to H^1( B_{r_0})$
and a constant $C>0$ such that, for every
$u\in H^1( B_{r_0}\setminus \overline{\Sigma})$,
\begin{align}
  &
(\mathsf{E} u)\big|_{B_{r_0}\setminus
    \overline{\Sigma}}=u, \label{ineq_extension_H0}\\
  &\|\mathsf{E}  u\|^2_{H^1(B_{r_0})} \le C \|u\|_{ H^1(B_{r_0}
    \setminus
    \overline{\Sigma})}^2, \label{ineq_extension_H1}\\
  &\|\nabla (\mathsf{E}  u)\|^2_{L^2(B_{r_0})} \le C
    \|\nabla u\|_{L^2(B_{r_0}
    \setminus
    \overline{\Sigma})}^2.\label{ineq_extension_gradients}
\end{align}
\end{lemma}
\begin{proof}
Let  $\mathsf{E}: H^1(B_{r_0} \setminus \overline{\Sigma})
\to H^1(B_{r_0})$ be the linear and bounded extension
operator defined as
\begin{equation*}
  \mathsf{E} u=
  \begin{cases}
    u,&\text{in }B_{r_0} \setminus \overline{\Sigma},\\
    v,&\text{in }\Sigma,
  \end{cases}
\end{equation*}
where $v\in H^1(\Sigma)$ is the unique harmonic function inside
$\Sigma$ whose trace on $\partial\Sigma$ coincides with the trace of
$u$.  In particular, $\mathsf{E}$ satisfies \eqref{ineq_extension_H0}
and \eqref{ineq_extension_H1}. Moreover, $\mathsf{E}u\equiv\gamma$ in
$B_{r_0}$ if $u\equiv\gamma$ in $B_{r_0} \setminus \overline{\Sigma}$
for some constant $\gamma\in\R$.

Assume by contradiction that \eqref{ineq_extension_gradients} does not
hold for any choice of the constant $C$.  Then there exists a sequence
$\{u_n\}_{n \in \mathbb{N}\setminus\{0\}}\subset H^1( B_{r_0}\setminus
\overline{\Sigma})$ such that
\begin{equation*}
  \|\nabla (\mathsf{E}  u_n)\|^2_{L^2(B_{r_0})} >n
  \|\nabla u_n\|_{L^2(B_{r_0}
    \setminus
    \overline{\Sigma})}^2\quad\text{for all }n\in\N \setminus\{0\}.
\end{equation*}
Let us define
\begin{equation*}
w_n:=\frac{u_n}{ \|\nabla (\mathsf{E} u_n)\|_{L^2(B_{r_0})}}.
\end{equation*}
Then 
\begin{equation*}
\norm{\nabla w_n}^2_{L^2(B_{r_0}
    \setminus
    \overline{\Sigma}) } < \frac{1}{n},
\end{equation*}
while 
\begin{equation*}
\|\nabla (\mathsf{E}  w_n)\|_{L^2(B_{r_0})}=1
\end{equation*}
for every $n \in \mathbb{N} \setminus \{0\}$.  By connectedness of
$B_{r_0} \setminus \overline{\Sigma}$ and the Poincar\'e-Wirtinger
inequality, see for example \cite[Theorem 13.27]{L_book_sobolev},
\begin{equation*}
\int_{B_{r_0} \setminus   \overline{\Sigma}} |w_n -\gamma_n|^2 \dx
\le
\kappa \int_{B_{r_0} \setminus   \overline{\Sigma}} |\nabla w_n|^2 \dx, 
\end{equation*}
for some $\kappa>0$ depending only on $B_{r_0}$ and $\Sigma$, where 
\begin{equation*}
  \gamma_n:= \frac1{|B_{r_0} \setminus   \overline{\Sigma}|_N}
  \int_{B_{r_0} \setminus   \overline{\Sigma}} w_n \dx.
\end{equation*}
Since $\nabla (\mathsf{E} \gamma_n)\equiv 0$ in $B_{r_0}$, it follows
that, for every $n \in \mathbb{N}\setminus\{0\}$,
\begin{align*}
  1&=\|\nabla (\mathsf{E}  w_n)\|_{L^2(B_{r_0})}^2=
  \|\nabla (\mathsf{E} ( w_n-\gamma_n))\|_{L^2(B_{r_0})}^2
  \le
\|\mathsf{E} ( w_n-\gamma_n)\|_{H^1(B_{r_0})}^2\\
&\le C \|w_n-\gamma_n\|^2_{H^1(B_{r_0}\setminus \overline{\Sigma})}
\le C(1+\kappa) \int_{B_{r_0}\setminus \overline{\Sigma} } |\nabla
w_n|^2 \dx
\le \frac{C(1+\kappa)}{n},
\end{align*}
a contradiction.
\end{proof}

\begin{proof}[Proof of Proposition \ref{prop_extension}]
 Thanks to an easy translation argument, it is not restrictive to
 assume that $x_0=0$, so that $\overline{B_{r_0}}\subset\Omega$,
 $\overline{\Sigma} \subset B_{r_0}$, and
 $\Sigma_\e=\e\Sigma$. For every $\e\in(0,1]$ and $u\in H^1(\Omega \setminus
 \overline{\Sigma}_\e)$ we define
 \begin{equation*}
   \Ee u(x)=
   \begin{cases}
     \mathsf{E}
     \big(u(\e\cdot)\big|_{B_{r_0}\setminus\overline{\Sigma}}\big)(x/\e),&\text{if }
       x\in \e B_{r_0},\\
     u(x),&\text{if }
       x\in \Omega\setminus \e\overline B_{r_0},
   \end{cases}
 \end{equation*}
 where
 $\mathsf{E}: H^1(B_{r_0} \setminus \overline{\Sigma}) \to
 H^1(B_{r_0})$ is the extension operator provided in Lemma
 \ref{l:extension_1}.  Hence, for every $\e\in(0,1]$,
 $\Ee: H^1(\Omega \setminus \overline{\Sigma}_\e)\to H^1(\Omega)$ is a
 well-defined linear operator satisfying \eqref{eq_Ee} in view of
 \eqref{ineq_extension_H0}. For every
 $u\in H^1(\Omega \setminus \overline{\Sigma}_\e)$, two changes of
 variables and \eqref{ineq_extension_H1} yield
 \begin{align}
   \notag  \|\Ee u\|_{L^2(\Omega)}^2&=   \int_{\Omega\setminus\e
                                      \overline{B_{r_0}}}|u|^2\dx+
                                      \int_{\e B_{r_0}}|\Ee u|^2\dx\\
   \notag & =\int_{\Omega\setminus\e
            \overline{B_{r_0}}}|u|^2\dx+\e^N
            \int_{B_{r_0}}|\mathsf{E} u_\e|^2\dx\\
   \notag&\leq \int_{\Omega\setminus\e
           \overline{B_{r_0}}}|u|^2\dx+\e^N C
           \int_{B_{r_0}\setminus\overline{\Sigma}}\big(|\nabla
           u_\e|^2+u_\e^2\big)\dx
   \\
   \label{eq:Ee1}&= \int_{\Omega\setminus\e
   \overline{B_{r_0}}}|u|^2\dx+ C
     \int_{\e B_{r_0}\setminus\overline{\Sigma_\e}}\big(\e^2|\nabla
     u|^2+u^2\big)\dx\leq (C+1)\|u\|_{H^1(\Omega\setminus\overline{\Sigma_\e})}^2
 \end{align}
 where $u_\e(x):=u(\e x)$.  Furthermore, by two change of variables and
 \eqref{ineq_extension_gradients},
 \begin{align}
   \notag  \|\nabla(\Ee u)\|_{L^2(\Omega)}^2&=   \int_{\Omega\setminus\e
                                              \overline{B_{r_0}}}|\nabla u|^2\dx+
                                              \int_{\e B_{r_0}}|\nabla(\Ee u)|^2\dx\\
   \notag &=   \int_{\Omega\setminus\e
            \overline{B_{r_0}}}|\nabla u|^2\dx+
            \e^{N-2}\int_{B_{r_0}}|\nabla(\mathsf{E} u_\e)|^2\dx\\
                                            &\leq  \int_{\Omega\setminus\e
                                              \notag \overline{B_{r_0}}}|\nabla u|^2\dx+C
                                              \e^{N-2}\int_{B_{r_0}\setminus\overline{\Sigma}}
                                              |\nabla u_\e|^2\dx\\
  \label{eq:Ee2} &= \int_{\Omega\setminus\e
   \overline{B_{r_0}}}|\nabla u|^2\dx+C
                              \int_{\e
     B_{r_0}\setminus\overline{\Sigma_\e}}|\nabla u|^2\dx
     \leq (C+1)\|\nabla u\|_{L^2(\Omega\setminus\overline{\Sigma_\e})}^2.
 \end{align}
 The uniform estimate \eqref{ineq_extension} then follows from
 \eqref{eq:Ee1}--\eqref{eq:Ee2}.
\end{proof}

\begin{remark}
  The assumption that the set $B_{r_0}\setminus\overline{\Sigma}$ is
  connected is necessary for the validity of Proposition
  \ref{prop_extension}, as shown in the following example.  Let
  $x_0=0\in\Omega$ and $r_0>0$ be such that
  $\overline{B_{r_0}}\subset\Omega$.  Let us consider an annular hole
  $\Sigma:= B_{r_0/2}\setminus \overline{ B_{r_0/4}}$, which clearly
  violates the connectedness condition of
  $B_{r_0}\setminus\overline{\Sigma}$.  Let us assume, for the sake of
  contradiction, that there exists a family of uniformly bounded
  extension operators
  $\Ee:H^1(\Omega \setminus \overline{\Sigma}_\e) \to H^1(\Omega)$
  satisfying \eqref{eq_Ee} and \eqref{ineq_extension} for every
  $\e\in(0,1]$, where $\Sigma_\e=\e\Sigma$.  Let
\begin{equation*}
u_\e(x):=
\begin{cases}
1, &\text{ if } x \in B_{r_0\e/4},\\
0, &\text{ if } x \in \Omega \setminus \overline{B_{r_0\e/2}}.
\end{cases}
\end{equation*}
Then, letting $\omega_N$ be the Lebesgue measure of the $N$-dimensional unit ball,
\begin{align}
\notag\frac{r_0^N\omega_N}{4^N}\e^N  &=|B_{r_0\e/4}|_N
  \\
  \label{eq:contradicts}&=\norm{u_\e}^2_{H^1(\Omega \setminus \overline{\Sigma}_\e)}\ge
\frac{1}{K} \int_{B_{r_0\e/2}} |\nabla (\Ee u_\e)|^2 \dx \ge
\frac{1}{K}
 \mathop{\rm Cap}\nolimits_{B_{r_0\e/2}}\left(\overline{B_{r_0\e/4}} \right),
\end{align}
for every $\e\in(0,1]$, where
$ \mathop{\rm Cap}_{B_{r_0\e/2}}\left(\overline{B_{r_0\e/4}} \right)$
is the classical condenser capacity of $\overline{B_{r_0\e/4}}$ in
$B_{r_0\e/2}$.
A direct computation yields
  \begin{equation*}
     \mathop{\rm Cap}_{B_{r_0\e/2}}\left(\overline{B_{r_0\e/4}} \right)=
    \begin{cases}
      \dfrac{N (N-2)\omega_N
        r_0^{N-2}}{4^{N-2}-2^{N-2}}\,\e^{N-2},&\text{if }N\geq
      3,\\[10pt]
      \dfrac{2\pi}{\log 2},&\text{if }N=2,
    \end{cases}
  \end{equation*}
 thus contradicting \eqref{eq:contradicts}. 
\end{remark}

\bibliographystyle{acm}

\begin{thebibliography}{10}

\bibitem{AFHC_spec_AB}
{\sc Abatangelo, L., Felli, V., Hillairet, L., and L\'{e}na, C.}
\newblock Spectral stability under removal of small capacity sets and
  applications to {A}haronov-{B}ohm operators.
\newblock {\em J. Spectr. Theory 9}, 2 (2019), 379--427.

\bibitem{AFN_fractional}
{\sc Abatangelo, L., Felli, V., and Noris, B.}
\newblock On simple eigenvalues of the fractional {L}aplacian under removal of
  small fractional capacity sets.
\newblock {\em Commun. Contemp. Math. 22}, 8 (2020), 1950071, 32.

\bibitem{AFT2014}
{\sc Abatangelo, L., Felli, V., and Terracini, S.}
\newblock On the sharp effect of attaching a thin handle on the spectral rate
  of convergence.
\newblock {\em J. Funct. Anal. 266}, 6 (2014), 3632--3684.

\bibitem{ALM_cap_exp}
{\sc Abatangelo, L., L\'{e}na, C., and Musolino, P.}
\newblock Ramification of multiple eigenvalues for the {D}irichlet-{L}aplacian
  in perforated domains.
\newblock {\em J. Funct. Anal. 283}, 12 (2022), Paper No. 109718, 50.

\bibitem{ALM24}
{\sc Abatangelo, L., L\'ena, C., and Musolino, P.}
\newblock Asymptotic behavior of generalized capacities with applications to
  eigenvalue perturbations: the higher dimensional case.
\newblock {\em Nonlinear Anal. 238\/} (2024), Paper No. 113391, 34.

\bibitem{AO2023}
{\sc Abatangelo, L., and Ognibene, R.}
\newblock Sharp behavior of {D}irichlet-{L}aplacian eigenvalues for a class of
  singularly perturbed problems.
\newblock {\em SIAM J. Math. Anal. 56}, 1 (2024), 474--500.

\bibitem{Aronson:95}
{\sc Aronson, R.}
\newblock Boundary conditions for diffusion of light.
\newblock {\em J. Opt. Soc. Am. A 12}, 11 (1995), 2532--2539.

\bibitem{Arridge20151033}
{\sc Arridge, S.~R., Kaipio, J.~P., Kolehmainen, V., and Tarvainen, T.}
\newblock Optical imaging.
\newblock In {\em Handbook of Mathematical Methods in Imaging: Volume 1, Second
  Edition}. Springer New York, 2015, pp.~1033--1079.

\bibitem{BS2023}
{\sc Barseghyan, D., and Schneider, B.}
\newblock Spectral convergence of the {L}aplace operator with {R}obin boundary
  conditions on a small hole.
\newblock {\em Mediterr. J. Math. 20}, 6 (2023), Paper No. 304, 18.

\bibitem{BS_quantitative}
{\sc Bisterzo, A., and Siclari, G.}
\newblock Quantitative spectral stability for compact operators.
\newblock {\em Preprint \tt{arXiv:2407.20809}\/} (2024).

\bibitem{BFK_survey}
{\sc Bucur, D., Freitas, P., and Kennedy, J.}
\newblock The {R}obin problem.
\newblock In {\em Shape optimization and spectral theory}. De Gruyter Open,
  Warsaw, 2017, pp.~78--119.

\bibitem{BGT_domain_pert}
{\sc Bucur, D., Giacomini, A., and Trebeschi, P.}
\newblock The {R}obin-{L}aplacian problem on varying domains.
\newblock {\em Calc. Var. Partial Differential Equations 55}, 6 (2016), Art.
  133, 29.

\bibitem{CCH_alpha}
{\sc Cakoni, F., Chaulet, N., and Haddar, H.}
\newblock On the asymptotics of a {R}obin eigenvalue problem.
\newblock {\em C. R. Math. Acad. Sci. Paris 351}, 13-14 (2013), 517--521.

\bibitem{ColindeV1986}
{\sc Colin~de Verdi\`ere, Y.}
\newblock Sur la multiplicit\'{e} de la premi\`ere valeur propre non nulle du
  laplacien.
\newblock {\em Comment. Math. Helv. 61}, 2 (1986), 254--270.

\bibitem{C_mainfols_holes}
{\sc Courtois, G.}
\newblock Spectrum of manifolds with holes.
\newblock {\em J. Funct. Anal. 134}, 1 (1995), 194--221.

\bibitem{DD1997}
{\sc Dancer, E.~N., and Daners, D.}
\newblock Domain perturbation for elliptic equations subject to {R}obin
  boundary conditions.
\newblock {\em J. Differential Equations 138}, 1 (1997), 86--132.

\bibitem{DSJ_book}
{\sc Dunford, N., and Schwartz, J.~T.}
\newblock {\em Linear operators. {P}art {II}}.
\newblock Wiley Classics Library. John Wiley \& Sons, Inc., New York, 1988.
\newblock Spectral theory. Selfadjoint operators in Hilbert space, With the
  assistance of William G. Bade and Robert G. Bartle, Reprint of the 1963
  original, A Wiley-Interscience Publication.

\bibitem{FLO_Neumann}
{\sc Felli, V., Liverani, L., and Ognibene, R.}
\newblock Quantitative spectral stability for the {N}eumann {L}aplacian in
  domains with small holes.
\newblock {\em J. Funct. Anal. 288}, 6 (2025), Paper No. 110817.

\bibitem{FNO_disa_Dirichlet_region}
{\sc Felli, V., Noris, B., and Ognibene, R.}
\newblock Eigenvalues of the {L}aplacian with moving mixed boundary conditions:
  the case of disappearing {D}irichlet region.
\newblock {\em Calc. Var. Partial Differential Equations 60}, 1 (2021), Paper
  No. 12, 33.

\bibitem{FNO_disa_Neuman_region}
{\sc Felli, V., Noris, B., and Ognibene, R.}
\newblock Eigenvalues of the {L}aplacian with moving mixed boundary conditions:
  the case of disappearing {N}eumann region.
\newblock {\em J. Differential Equations 320\/} (2022), 247--315.

\bibitem{FNOS_AB_12}
{\sc Felli, V., Noris, B., Ognibene, R., and Siclari, G.}
\newblock Quantitative spectral stability for {A}haronov-{B}ohm operators with
  many coalescing poles.
\newblock {\em J. Eur. Math. Soc., published online\/} (2025).

\bibitem{FNS_AB}
{\sc Felli, V., Noris, B., and Siclari, G.}
\newblock On {A}haronov-{B}ohm operators with multiple colliding poles of any
  circulation.
\newblock {\em Nonlinear Anal., to appear\/}.

\bibitem{FO2020}
{\sc Felli, V., and Ognibene, R.}
\newblock Sharp convergence rate of eigenvalues in a domain with a shrinking
  tube.
\newblock {\em J. Differential Equations 269}, 1 (2020), 713--763.

\bibitem{FR_polyharmonic}
{\sc Felli, V., and Romani, G.}
\newblock Perturbed eigenvalues of polyharmonic operators in domains with small
  holes.
\newblock {\em Calc. Var. Partial Differential Equations 62}, 4 (2023), Paper
  No. 128, 36.

\bibitem{F_alpha}
{\sc Filinovskiy, A.~V.}
\newblock On the asymptotic behavior of eigenvalues and eigenfunctions of the
  {R}obin problem with large parameter.
\newblock {\em Math. Model. Anal. 22}, 1 (2017), 37--51.

\bibitem{jimbo}
{\sc Jimbo, S.}
\newblock Eigenvalues of the {Laplacian} in a domain with a thin tubular hole.
\newblock {\em J. Elliptic Parabol. Equ. 1}, 1 (2015), 137--174.

\bibitem{LdC2012}
{\sc Lanza~de Cristoforis, M.}
\newblock Simple {Neumann} eigenvalues for the {Laplace} operator in a domain
  with a small hole.
\newblock {\em Rev. Mat. Complut. 25}, 2 (2012), 369--412.

\bibitem{nazarov2011}
{\sc Laurain, A., Nazarov, S., and Sokolowski, J.}
\newblock Singular perturbations of curved boundaries in three dimensions.
  {The} spectrum of the {Neumann} {Laplacian}.
\newblock {\em Z. Anal. Anwend. 30}, 2 (2011), 145--180.

\bibitem{L_book_sobolev}
{\sc Leoni, G.}
\newblock {\em A first course in {S}obolev spaces}, second~ed., vol.~181 of
  {\em Graduate Studies in Mathematics}.
\newblock American Mathematical Society, Providence, RI, 2017.

\bibitem{LP_alpha}
{\sc Levitin, M., and Parnovski, L.}
\newblock On the principal eigenvalue of a {R}obin problem with a large
  parameter.
\newblock {\em Math. Nachr. 281}, 2 (2008), 272--281.

\bibitem{McG_capacity}
{\sc McGillivray, I.}
\newblock Capacitary asymptotic expansion of the groundstate to second order.
\newblock {\em Comm. Partial Differential Equations 23}, 11-12 (1998),
  2219--2252.

\bibitem{O_Robin}
{\sc Ognibene, R.}
\newblock On asymptotics of {R}obin eigenvalues in the {D}irichlet limit.
\newblock {\em Preprint \tt{arXiv:2407.19505}\/} (2024).

\bibitem{O_Dirichlet_holes_euclidean}
{\sc Ozawa, S.}
\newblock Singular variation of domains and eigenvalues of the {L}aplacian.
\newblock {\em Duke Math. J. 48}, 4 (1981), 767--778.

\bibitem{ozawa1983}
{\sc Ozawa, S.}
\newblock Spectra of domains with small spherical {Neumann} boundary.
\newblock {\em J. Fac. Sci., Univ. Tokyo, Sect. I A 30\/} (1983), 259--277.

\bibitem{ozawa1985}
{\sc Ozawa, S.}
\newblock Asymptotic property of an eigenfunction of the {Laplacian} under
  singular variation of domains - the {Neumann} condition -.
\newblock {\em Osaka J. Math. 22\/} (1985), 639--655.

\bibitem{ozawa1}
{\sc Ozawa, S.}
\newblock Singular variation of domain and spectra of the {L}aplacian with
  small {R}obin conditional boundary. {I}.
\newblock {\em Osaka J. Math. 29}, 4 (1992), 837--850.

\bibitem{ozawa2}
{\sc Ozawa, S., and Roppongi, S.}
\newblock Singular variation of domain and spectra of the {L}aplacian with
  small {R}obin conditional boundary. {II}.
\newblock {\em Kodai Math. J. 15}, 3 (1992), 403--429.

\bibitem{rauch_taylor:75}
{\sc Rauch, J., and Taylor, M.}
\newblock Potential and scattering theory on wildly perturbed domains.
\newblock {\em J. Functional Analysis 18\/} (1975), 27--59.

\bibitem{SW_extension}
{\sc Sauter, S.~A., and Warnke, R.}
\newblock Extension operators and approximation on domains containing small
  geometric details.
\newblock {\em East-West J. Numer. Math. 7}, 1 (1999), 61--77.

\bibitem{ward-henshaw-keller}
{\sc Ward, M.~J., Henshaw, W.~D., and Keller, J.~B.}
\newblock Summing logarithmic expansions for singularly perturbed eigenvalue
  problems.
\newblock {\em SIAM J. Appl. Math. 53}, 3 (1993), 799--828.

\bibitem{ward-keller}
{\sc Ward, M.~J., and Keller, J.~B.}
\newblock Strong localized perturbations of eigenvalue problems.
\newblock {\em SIAM J. Appl. Math. 53}, 3 (1993), 770--798.

\end{thebibliography}

\end{document}